[1994/06/01]% LaTeX date must June 1994 or later

\documentclass[12pt,reqno,intlimits,english,centertags]{amsart}

\usepackage[T1]{fontenc}
\usepackage{ifthen,calc}
\usepackage{babel}
\usepackage{xspace}
\usepackage{amsmath}
\usepackage{lmodern}
\usepackage{amsthm}
\numberwithin{equation}{section}
\theoremstyle{plain}
\newtheorem{theorem}{Theorem}[section]
\newtheorem{lemma}[theorem]{Lemma}
\newtheorem{corollary}[theorem]{Corollary}

\theoremstyle{remark}
\newtheorem{remark}[theorem]{Remark}
\theoremstyle{definition}

\newcommand{\abs}[1]{\lvert#1\rvert}
\DeclareMathOperator{\Lapl}{\Delta}
\newcommand{\norm}[1]{\lVert#1\rVert}
\newcommand{\norma}[2]{\norm{#1}_{#2}}
\newcommand{\SpDim}{N}
\newcommand{\numberspacefont}{\boldsymbol}
\newcommand{\R}{\numberspacefont{R}}
\newcommand{\RN}{\R^{\SpDim}}

\newcommand{\Pposbase}[3][*]{\ifthenelse{\equal{#1}{*}}%
{(#2)_{#3}}{\left(#2\right)_{#3}}}
\newcommand{\ppos}[1]{\Pposbase[*]{#1}{+}}

\newcommand{\pneg}[1]{\Pposbase[*]{#1}{-}}

\newcommand{\di}{\,\textup{\textmd{d}}}
\newcommand{\eps}{\varepsilon}

\newcommand{\unk}{u}

\newcommand{\pder}[2]{\frac{\partial #1}{\partial #2}}

\newcommand{\bdr}[1]{\partial #1}

\newcommand{\grad}{\operatorname{\nabla}}
\DeclareMathOperator{\supp}{supp}
\newcommand{\der}[2]{\frac{\di #1}{\di #2}}

\newcommand{\vol}{V}
\newcommand{\dnf}{\rho}
\newcommand{\msr}{\mu}
\newcommand{\msd}{\mu_{\dnf}}
\newcommand{\ipf}{\omega}

\newcommand{\iso}{h}
\newcommand{\isoexp}{\frac{\SpDim-1}{\SpDim}}
\newcommand{\lov}{R}
\newcommand{\vold}{\vol_{\dnf}}
\newcommand{\dlov}{\lov_{\dnf}}
\newcommand{\measd}{\nu_{\dnf}}
\newcommand{\fspf}{\tilde Z}
\newcommand{\fpsf}{\psi}
\newcommand{\Ha}{\mathcal{H}}
\newcommand{\Ka}{\mathcal{K}}
\newcommand{\arn}{\omega_\SpDim}
\newcommand{\smo}{\textup{\textmd{o}}}
\begin{document}

%%%% amsart
\title
[Equations on inhomogeneous manifolds]{Asymptotic properties of solutions to the Cauchy problem for
  degenerate parabolic equations with inhomogeneous density on manifolds}%
%% ppt
\author{Daniele Andreucci}
\address{Department of Basic and Applied Sciences for Engineering\\Sapienza University of Rome\\via A. Scarpa 16 00161 Rome, Italy}
\email{daniele.andreucci@sbai.uniroma1.it}
\thanks{The first author is member of the Gruppo Nazionale
  per la Fisica Matematica (GNFM) of the Istituto Nazionale di Alta Matematica
  (INdAM)}
\author{Anatoli F. Tedeev}
\address{South Mathematical Institute of VSC RAS\\Vladikavkaz, Russian Federation}
\email{a\_tedeev@yahoo.com}
\thanks{The second author was supported by Sapienza Grant C26V17KBT3}
%%%% 

%%%% 

%%%% amsart
\date{\today}
%%%% 
\begin{abstract}
  We consider the Cauchy problem for doubly nonlinear degenerate
  parabolic equations with inhomogeneous density on noncompact
  Riemannian manifolds. We give a qualitative classification of the
  behavior of the solutions of the problem depending on the behavior
  of the density function at infinity and the geometry of the
  manifold, which is described in terms of its isoperimetric
  function. We establish for the solutions properties as:
  stabilization of the solution to zero for large times, finite speed
  of propagation, universal bounds of the solution, blow up of the
  interface. Each one of these behaviors of course takes place in a
  suitable range of parameters, whose definition involves a
  universal geometrical characteristic function, depending both on the
  geometry of the manifold and on the asymptotics of the density at infinity.
\end{abstract}

%%%% amsart
\maketitle
%%%% 

\section{Introduction}\label{s:intro}
\label{s:stat}
We consider the Cauchy problem
\begin{alignat}{2}
  \label{eq:pde}
  \dnf(x)
  \unk_{t}
  -
  \Lapl_{p,m}
  (\unk)
  &=
  0
  \,,
  &\qquad&
  x\in M
  \,,
  t>0
  \,,
  \\
  \label{eq:initd}
  \unk(x,0)
  &=
  \unk_{0}(x)
  \,,
  &\qquad&
  x\in M
  \,.
\end{alignat}
Here $M$ is a complete Riemannian manifold of topological dimension
$\SpDim$, with infinite volume. In local coordinates
$x^{i}$, we denote
\begin{equation*}
  \Lapl_{p,m}(u)
  =
  \frac{
    1
  }{
    \sqrt{\det(g_{ij})}
  }
  \sum_{i,j=1}^{\SpDim}
  \pder{}{{x^{i}}}
  \Big(
  \sqrt{\det(g_{ij})}
  g^{ij}
  \unk^{m-1}
  \abs{\grad\unk}^{p-2}
  \pder{\unk}{x^{j}}
  \Big)
  \,,
\end{equation*}
where $(g_{ij})$ denotes the Riemannian metric, $(g^{ij})=(g_{ij})^{-1}$,
$\grad \unk$ is the gradient with respect to $(g_{ij})$, and
\begin{equation*}
  \abs{\grad u}^{2}
  =
  \sum_{i,j=1}^{\SpDim}
  g^{ij}
  \pder{\unk}{{x^{j}}}
  \pder{\unk}{{x^{i}}}
  \,.
\end{equation*}

We always assume that $1<p<\SpDim$, and that either
\begin{equation}
  \label{eq:pde_deg}
  p+m-3>0
  \,,
\end{equation}
or
\begin{equation}
  \label{eq:pde_sin}
  p+m-3<0
  \,,
  \qquad
  m>0
  \,.
\end{equation}

In this paper we follow an approach ultimately based on the classical DeGiorgi estimates (\cite{DiB:dpe,LSU}); a new technical tool is the weighted Faber-Krahn type inequality of Lemmas \ref{l:fk} and \ref{l:fks}. This inequality takes into account the asymptotic behaviors both of the density function $\dnf$ and of the volume growth of the manifold at infinity. We use the isoperimetrical properties of the manifold, which also allow us to prove new embedding results which we think are of independent
interest.

We establish in the slow diffusion case \eqref{eq:pde_deg}, and also in the fast diffusion case \eqref{eq:pde_sin} under additional assumptions, the decay rate for large times of nonnegative solutions, for initial data of finite mass. In the degenerate case we also estimate the finite speed of propagation for the support of solutions with a bounded support. These results apply in a subcritical case where, roughly speaking, the density function decays not too fast at infinity. Where we have explicit solutions, that is in the Euclidean case, our estimates reduce to the known optimal ones.

Still in the slow diffusion case, but when the density function decays fast enough, we investigate the behavior of the solutions for large $t$; we obtain under different assumptions a universal bound for solutions and a result of interface blow up. The universal bound is the same as in the Euclidean case, as expected; see also Subsection~\ref{s:exam}.

The various cases recalled above are discriminated in terms of the
behavior of a universal function involving the density and the volume
growth of the manifold (see Remark~\ref{r:decayf}).

The interest of this problem appeared first in the case $M=\R^{3}$ with
the Euclidean metric, where \cite{Kamin:Rosenau:1981},
\cite{Rosenau:Kamin:1982} obtained the first surprising results, in symmetric
cases, on the qualitative properties of solutions to the
porous media equation with inhomogeneous density. The interface blow up in the Euclidean setting
was first shown in \cite{Kamin:Kersner:1993};
\cite{Tedeev:2007} extended it to a wide class of doubly degenerate parabolic
equations.

Let us here explicitly recall the behavior of solutions in the
Euclidean case, when $\dnf(x)=(1+\abs{x})^{-\alpha}$, $x\in\RN$, for a
given $0<\alpha\le \SpDim$. Such a behavior strongly depends on the
interplay between the nonlinearities appearing in the equation, at
least from a two-fold point of view, that is both as far as the sup
bounds for solutions are concerned, and then also if we consider the
property of finite speed of propagation, see \cite{Tedeev:2007} for
the following results; see also \cite{Dzagoeva:Tedeev:2018}.
\\
Concerning the first issue, the limiting threshold is $\alpha=p$: in
the range $\alpha\le p$ the sup estimates for solutions keep the same
structure as the ones known for the homogeneous doubly nonlinear
equation (say with $\alpha=0$ above), although of course the exact
dependence on the parameters of the problem is not the same. More
explicitly, the solution is bounded above by the product of a negative
power of time and a certain positive power of the initial
mass. Instead in the range $\alpha>p$ one can prove a universal bound:
a suitable negative power of time still is present, but the initial
mass disappears from the estimate.
\\
Concerning the second issue of finite speed of propagation, which is
connected to conservation of mass, a second explicit threshold
$\alpha_{*}\in(p,\SpDim)$ appears. We assume clearly that the initial
data has compact support. Then in the subcritical range
$\alpha<\alpha_{*}$ the property of finite speed of propagation is
preserved for all times, i.e., the solution is compactly supported for
$t>0$. As a result, also the property of conservation of mass is valid
for all times. In the supercritical range $\alpha>\alpha_{*}$ the
evolution of the support is quite different: finite
speed of propagation (that is boundedness of support) and
conservation of mass can not hold true for all $t>0$.

In Subsection~\ref{s:exam} below we present a class of examples and a
more detailed comparison with the Euclidean case. Before describing
the results of this paper, we recall that parabolic
problems in a Euclidean metric with
inhomogeneous density were studied in \cite{Martynenko:Tedeev:2007},
\cite{Martynenko:Tedeev:2008} (blow up phenomena); \cite{Reyes:Vazquez:2009},
\cite{Kamin:Reyes:Vazquez:2010} (asymptotic expansion of the
solution of the porous media equation); \cite{Nieto:Reyes:2013},
\cite{Iagar:Sanchez:2013} (critical case).  We
quote for results related to ours
\cite{Grillo:Muratori:2013,Grillo:Muratori:Porzio:2013,Punzo:2011} for
the porous media equation and \cite{Duzgun:Mosconi:Vespri:2019} for
anisotropic operators. Still on the subject of porous media like
equations on Riemannian manifolds, besides the seminal papers
\cite{Punzo:2012b, Punzo:2012c} on the properties of the support of
solutions, we also quote \cite{Grillo:Muratori:2016,Vazquez:2015}.

The main goal of the present paper is to investigate
the behavior of solutions around the threshold discriminating between
the cases described above for the sup estimates, in terms of the
density function $\dnf$, the nonlinearities in the equation, and of
course the Riemannian geometry of $M$.
\\
Our results on the subcritical
$L^{\infty}$ bound in fact apply when we, clearly, are in the
subcritical case, but not too far from the threshold. In the Euclidean
example just discussed, i.e., $\alpha$ should be not too small, which
is not restrictive in the light of the purpose of this paper. On the
other hand, we provide a unified approach to both the degenerate
\eqref{eq:pde_deg} and the singular \eqref{eq:pde_sin} cases. The case
where we are far below the threshold calls for a different approach;
this will be the subject of a forthcoming paper. 
\\
Such a requirement of closeness to the threshold is not needed by the other results.
\\
See also
\cite{Andreucci:Cirmi:Leonardi:Tedeev:2001} for the Euclidean case; we
borrow the energetic setting of
\cite{Andreucci:Tedeev:1999,Andreucci:Tedeev:2005,Andreucci:Tedeev:2015,Andreucci:Tedeev:2017}; see also \cite{Tedeev:1993}.

\subsection{Assumptions.}
\label{s:ass_main}

In what follows $d(x)$ denotes the geodesic distance of $x$ from a fixed point $x_{0}\in M$, $\abs{U}$ is the Riemannian volume of $U\subset M$, and $\abs{\bdr{U}}_{\SpDim-1}$ the corresponding area of its boundary.

An important role is played by the function
\begin{equation*}
  \vol(R)
  =
  \abs{B_{R}}
  \,,
  \qquad
  B_{R}
  =
  \{
  x\in M
  \mid
  d(x)\le R
  \}
  \,.
\end{equation*}
On the geometry of the manifold $M$ we need the following requirements, of isoperimetrical character. We assume that for all bounded and Lipschitz domains $U\subset M$
\begin{equation}
  \label{eq:M_iso}
  \abs{\bdr{U}}_{\SpDim-1}
  \ge
  \iso(\abs{U})
  \,,
  \qquad
  s\mapsto
  \ipf(s)
  =
  \frac{s^{\frac{\SpDim-1}{\SpDim}}}{\iso(s)}
  \text{ is nondecreasing in $(0,+\infty)$.}
\end{equation}
Here $h:[0,+\infty)\to[0,+\infty)$ is a continuos nondecreasing
function, with $h(0)=0$. Isoperimetrical inequalities are in some sense equivalent to embedding theorems, see \cite{Mazja:ss}.

We list here all the assumptions required in
the following; all of them are needed for the subcritical sup estimate, while the
other results employ only a subset of such hypotheses.

\subsubsection{Volume growth conditions.} We require the growth conditions
\begin{equation}
  \label{eq:M_grow}
  c
  \iso(\vol(R))
  \le
  \der{\vol}{R}(R)
  \,,
  \qquad
  R>0
  \,,
\end{equation}
and its counterpart
\begin{equation}
  \label{eq:M_growup}
  \der{\vol}{R}(R)
  \le
  c^{-1}
  \iso(\vol(R))
  \,,
  \qquad
  R>0
  \,,
\end{equation}
for a given $0<c<1$.

The following condition of non-parabolicity of the manifold is needed to prove global embedding results:
\begin{equation}
  \label{eq:M_phyp}
  \int_{0}^{k}
  \frac{\di t}{\vol^{(-1)}(t)^{p}}
  \le
  c^{-1}
  \frac{k}{\vol^{(-1)}(k)^{p}}
  \,,
  \qquad
  k>0
  \,.
\end{equation}

In some cases we need that $R^{\SpDim}/\vol(R)$ is nondecreasing, which is implied by the assumption
\begin{equation}
  \label{eq:M_inc}
  \der{\vol}{R}(R)
  \le
  \SpDim
  \frac{\vol(R)}{R}
  \,.
\end{equation}
In addition we require
\begin{equation}
  \label{eq:M_isoup}
  \frac{\vol(R)}{R}
  \le
  c^{-1}
  \iso(V(R))
  \,,
  \qquad
  R>0
  \,.
\end{equation}

\subsubsection{Density decay conditions.}
In order to get the subcritical sup estimate, the density function $\dnf$ is required to satisfy:
\begin{gather}
  \label{eq:dnf_dec}
  s\mapsto \dnf(s)s^{\alpha_{1}}
  \text{ is nonincreasing for $s\in(1,+\infty)$;}
  \\
  \label{eq:dnf_inc}
  s\mapsto \dnf(s)s^{\alpha_{2}}
  \text{ is nondecreasing for $s\in(1,+\infty)$,}
\end{gather}
where $0<\alpha_{1}<\alpha_{2}<p$ are given constants.

Note that \eqref{eq:M_phyp}, \eqref{eq:dnf_inc} imply that the function $\vold(R)=\dnf(R)\vol(R)$ is bounded from above and below by constant multiples of the same increasing function (see Lemma~\ref{l:aux_density} below); however we need state the more precise assumption
\begin{equation}
  \label{eq:dnf_vol}
  \vold(R)
  =
  \dnf(R)
  \vol(R)
  \text{ is increasing for $R\in(0,+\infty)$.}
\end{equation}
We denote the inverse function of $\vold$ by $\dlov$. This hypothesis will be assumed implicitly throughout.

Finally, as we remarked above, we need to be not too far from the threshold, in Theorem~\ref{t:sup};
exactly, this means that, on setting $p^{*}=\SpDim p/(\SpDim-p)$,
\begin{equation}
  \label{eq:close}
  s
  \mapsto
  \dnf(s)
  \ipf(V(s))^{p^{*}}
  \,,
  \text{ is nonincreasing in $s>s_{0}$,}
\end{equation}
for a suitably chosen $s_{0}>0$. 

\begin{remark}
  \label{r:phyp}
  One can see easily that assumption \eqref{eq:M_phyp} is a consequence of the following alternative assumption: $R\mapsto \vol(R)/R^{q}$ is nondecreasing for $R>0$ for some $q>p$.
  \\
  Assumption \eqref{eq:M_phyp} is not merely a technical device; if it fails we do not expect decay of solutions for large times, see \cite{Kamin:Rosenau:1981}.
\end{remark}

\begin{remark}
  \label{r:decayf}
  Consider the functions
  \begin{gather}
    \label{eq:decayf_b}
    \fpsf(R)
    =
    \vold(R)^{p+m-3}
    \dnf(R)
    R^{p}
    \,,
    \quad
    R>0
    \,;
    \\
    \label{eq:decayf_a}
    b_{1}(s)
    =
    \fpsf(\dlov(s))
    =
    s^{p+m-3}
    \dnf(\dlov(s))
    \dlov(s)^{p}
    \,,
    \quad
    s>0
    \,.
  \end{gather}
  Owing to our assumptions, $b_{1}$ is increasing if and only if $\fpsf$ is. In turn, this is automatically satisfied in the degenerate case \eqref{eq:pde_deg}, at least if \eqref{eq:dnf_inc} is assumed, but it is not necessarily valid in the singular case \eqref{eq:pde_sin}. However, when $\fpsf$ is increasing we denote its inverse by $\fspf$; the latter function provides in some cases an estimate of the finite speed of propagation of the support of a solution.
\end{remark}

The definition of weak solution to \eqref{eq:pde}--\eqref{eq:initd} is in fact standard; the problem is given the integral formulation
\begin{equation}
  \label{eq:weaksol}
  \int_{0}^{+\infty}
  \int_{M}
  \{
  -
  u
  \dnf
  \zeta_{t}
  +
  u^{m-1}
  \abs{\grad u}^{p-2}
  \grad u
  \grad \zeta
  \}
  \di\msr
  \di t
  =
  \int_{M}
  u_{0}
  \dnf
  \zeta(x,0)
  \di\msr
  \,,
\end{equation}
for all $\zeta\in C^{1}(M\times[0,+\infty))$, with bounded support.
In general, the existence can be proved following the methods of \cite{AubinNPRG,Bernis:1988a} in the setting of energy solutions, i.e., assuming $\sqrt{\dnf}u_{0}\in L^{2}(M)$. The solution obtained satisfies $u\in L^{\infty}_{\textup{loc}}(M\times(0,+\infty))$, $u\in C((0,T);L^{2}(M))$, $u^{m-1}\abs{\grad u}^{p}\in L^{1}(M\times (0,T))$ for all $T<+\infty$. In the subcritical cases where the assumptions of Theorem~\ref{t:sup} are in force, in fact we can prove existence even for initial data which are Radon measures, such that $\dnf u_{0}$ has finite mass in $M$. This follows from the estimate in Theorem~\ref{t:sup} and from a standard approximation procedure via a sequence of solutions to initial-value boundary problems, with vanishing Dirichlet data, in a sequence of invading compact domains. In the latter case however the solution satisfies $u\in L^{\infty}_{\textup{loc}}(M\times(0,+\infty))$, $u\in C((0,T);L^{2}_{\textup{loc}}(M))$, $u^{m-1}\abs{\grad u}^{p}\in L^{1}_{\textup{loc}}(M\times (0,T))$ for all $T<+\infty$. In addition, the proof requires an estimate of the $L^{1}$ norm of $u^{m-1}\abs{\grad u}^{p-1}$ up to time $t=0$ (sometimes called an entropy estimate); this bound can be achieved following \cite{Andreucci:Tedeev:1998,Andreucci:Tedeev:2000}, and again using Theorem~\ref{t:sup}.
\\
The uniqueness of solutions for problems of the kind we consider is well known to be a difficult problem in general; see e.g., the seminal paper \cite{Eidus:Kamin:1994}, and more recently \cite{Punzo:2012}.

Thus in the following theorems, we refer to a solution $u$ constructed by approximation as shown above. Also, in the proofs for simplicity we work sometimes with a strong almost everywhere formulation of the differential equation, to avoid the by now standard regularization arguments. We denote by $\gamma$, $\gamma_{0}$, \dots, constants (varying from line to line) depending only on the parameters of the problem.

\begin{remark}
  \label{r:mass_subcr}
  Since we can limit the $L^{1}(M)$ norm
  of each such approximation only in terms of the initial mass,
  passing to the limit we infer for the solution referred to just above
  \begin{equation}
    \label{eq:mass_subcr_a}
    \norma{\unk(t)\dnf}{L^{1}(M)}
    \le
    \gamma
    \norma{\unk_{0}\dnf}{L^{1}(M)}
    \,,
    \qquad
    0<t<+\infty
    \,.
  \end{equation}
  Here $\gamma$ depends on $M$ and $\dnf$, but not on $u$.
  Notice that this bound follows without assuming finite speed of
  propagation. However, it is easy to prove that if the support of the solution is bounded over $[0,T]$, mass is conserved exactly, that is
  \begin{equation}
    \label{eq:mass_subcr_ab}
    \norma{\unk(t)\dnf}{L^{1}(M)}
    =
    \norma{\unk_{0}\dnf}{L^{1}(M)}
    \,,
    \qquad
    0<t<T
    \,.
  \end{equation}
\end{remark}

\begin{remark}
  \label{r:dnf}
  It follows without difficulty from our arguments that the radial
  character and the assumptions on $\dnf$ can be replaced by
  analogous statements on a radial function $\tilde \dnf$ such that
  \begin{equation*}
    c
    \tilde\dnf(x)
    \le
    \dnf(x)
    \le
    c^{-1}
    \tilde \dnf(x)
    \,,
    \qquad
    x\in M
    \,,
  \end{equation*}
  for a given $0<c<1$.
\end{remark}

\subsection{Main results.}
\label{s:main}

We begin with the subcritical sup estimate.
\begin{theorem}
  \label{t:sup}
  Assume \eqref{eq:M_iso}--\eqref{eq:close}. We also assume one of the following: i) $p+m-3<0$, $m>0$,  $\fpsf$ as in \eqref{eq:decayf_b} is increasing with $\fpsf(0+)=0$, $\fpsf(R)\to+\infty$ as $R\to+\infty$, and 
  \begin{equation}
    \label{eq:sup_n}
    (\SpDim-\alpha_{1})(p+m-3)
    +
    p-\alpha_{2}
    >0
    \,;
  \end{equation}
  ii) $p+m-3>0$ (in this case the other conditions in i) are automatically satisfied).
  \\
  Let $\dnf u_{0}\in L^{1}(M)$. Then for $\fspf=\fpsf^{(-1)}$,
  \begin{equation}
    \label{eq:sup_m}
    \norma{u(t)}{\infty}
    \le
    \frac{
      \norma{\dnf u_{0}}{1}
    }{
      \vold\big(\fspf(\gamma_{0} t \norma{\dnf u_{0}}{1}^{p+m-3})\big)
    }
    \,,
    \qquad
    t>0
    \,,
  \end{equation}
  for a constant $\gamma_{0}>0$ independent of $u$.
\end{theorem}

Next we deal with the finite speed of propagation.

\begin{theorem}
  \label{t:fsp}
  Assume that $\supp u_{0}\subset B_{R_{0}}$ and that we are in the degenerate case \eqref{eq:pde_deg}, with $\fpsf$ increasing, $\fpsf(0+)=0$. Assume further \eqref{eq:M_iso} and that for a suitable $C>0$
  \begin{equation}
    \label{eq:fsp_m}
    \dnf(r)
    \le C
    \dnf(2r)
    \,,
    \qquad
    r\ge 0
    \,.
  \end{equation}
  Let $\dnf u_{0}\in L^{1}(M)$. Then for all $t>0$, $\supp u(t)\subset B_{R}$ if $R$ satisfies
  \begin{equation}
    \label{eq:fsp_n}
    R
    =
    4R_{0}
    +
    \fspf\big(
    \gamma
    t
    \norma{\dnf u_{0}}{1}^{p+m-3}
    \big)
    \,,
  \end{equation}
  for a constant $\gamma>0$ independent of $u$.
\end{theorem}

The reason why we can avoid in Theorem~\ref{t:fsp} several of the global assumptions of Theorem~\ref{t:sup} is that in the proof we work
in bounded sets, shaped like annuli. Thus we don’t need to
apply a weighted Sobolev inequality, which dispenses us from assuming non-parabolicity as in \eqref{eq:M_phyp}. By the same token, we need only a standard doubling property for $\dnf$.

In the supercritical case \eqref{eq:unb_nk} we can prove the following universal, or absolute, sup bound for the solution, which is in fact in its functional form independent of the initial data and of the geometry of $M$.
\begin{theorem}
  \label{t:unb}
  Assume that \eqref{eq:pde_deg}, \eqref{eq:M_iso}--\eqref{eq:M_isoup} hold true, and that for a suitable $c>0$
  \begin{equation}
    \label{eq:unb_nk}
    \dnf(\tau)
    \le
    c^{-1} \tau^{-\alpha}
    \,,
    \quad
    \tau>1
    \,,
  \end{equation}
  for some $\alpha>p$.
  \\
  Let $\dnf u_{0}\in L^{1}(M)$, $\sqrt{\dnf} u_{0}\in L^{2}(M)$. Then
  \begin{equation}
    \label{eq:unb_m}
    \norma{u(t)}{\infty}
    \le
    \gamma
    t^{-\frac{1}{p+m-3}}
    \,,
    \qquad
    t>0
    \,,
  \end{equation}
  where $\gamma$ does not depend on $u$.
\end{theorem}

Finally, if the second threshold is exceeded, which we essentially
assume in \eqref{eq:ibl_n}, the finite speed of propagation property
fails.
\begin{theorem}
  \label{t:ibl}
  Assume \eqref{eq:pde_deg}, \eqref{eq:M_iso}, \eqref{eq:M_phyp}, \eqref{eq:M_inc}, \eqref{eq:M_isoup}.
  Let $\dnf u_{0}\in L^{1}(M)$, $\sqrt{\dnf} u_{0}\in L^{2}(M)$
  and $\unk_{0}$ with bounded support. Assume that 
  \begin{gather}
    \label{eq:ibl_n}
    \int_{1}^{+\infty}
    \big(
    \tau^{p}
    \dnf(\tau)
    \big)^{r}
    \fpsf(\tau)^{\frac{1}{p+m-3}}
    \frac{\di\tau}{\tau}
    <
    +\infty
    \,,
  \end{gather}
  for $r\in(-r_{0},r_{0})$ for some $r_{0}>0$.
  \\
  Then the boundedness
  of the support of $\unk(t)$ fails over $(0,\bar t)$ for
  a sufficiently large $\bar t>0$.
\end{theorem}

\subsection{Example}
\label{s:exam}

Let us consider the case of a model Riemannian manifold (see e.g., \cite{Grigoryan:2006}) with metric $\di r^{2}+ f(r)^{2} \di \xi^{2}$, $r\ge 0$, $\xi\in S^{\SpDim-1}$. We have
\begin{equation*}
  \vol(R)
  =
  \arn
  \int_{0}^{R}
  f(s)^{\SpDim-1}
  \di s
  \,,
  \qquad
  \arn=\abs{S^{\SpDim-1}}_{\SpDim-1}
  \,.
\end{equation*}
We define the continuous functions $f$, $\dnf$ for $\tau\ge 0$ as in
\begin{equation*}
  f(\tau)
  =
  \left\{
    \begin{alignedat}{2}
      &C(A)\tau
      \,,
      &\quad&
      \tau\le A
      \,,
      \\
      &\tau^{\beta}
      (\ln\tau)^{\nu}
      \,,
      &\quad&
      \tau>A
      \,;
    \end{alignedat}
  \right.
  \quad
    \dnf(\tau)
  =
  \left\{
    \begin{alignedat}{2}
      &B^{-\alpha}(\ln B)^{\mu}
      \,,
      &\quad&
      \tau\le B
      \,,
      \\
      &\tau^{-\alpha}(\ln \tau)^{\mu}
      \,,
      &\quad&
      \tau>B
      \,,
    \end{alignedat}
    \right.
\end{equation*}
where $B\ge A\ge e$ are suitably chosen. Here $(p-1)/(\SpDim-1)<\beta< 1$ and $\mu$, $\nu\in \R$, $\alpha>0$. By means of lengthy but straightforward calculations one can check that our assumptions are satisfied, thereby implying the following.

The proof and the result of Theorem~\ref{t:sup} apply, for large $t$, in the two cases:
\\
i) $p+m-3<0$, $m>0$: we have to assume
\begin{gather}
  \label{eq:exa_n}
  \SpDim(p+m-3)+p>0
  \,,
  \\
  \label{eq:exa_nn}
  \alpha<
  \frac{\big(\beta(\SpDim-1)+1\big)(p+m-3)+p}{p+m-2}
  =:
  \alpha^{*}
  \,,
  \\
  \label{eq:exa_nnn}
  p>\alpha>\frac{\SpDim-1}{\SpDim}(1-\beta)p^{*}
  \,.
\end{gather}
ii) $p+m-3>0$: in this case \eqref{eq:exa_n} is
automatically satisfied. We assume  \eqref{eq:exa_nn}, \eqref{eq:exa_nnn}.
\\
It can be computed that the function $\fpsf$ takes here the
form
\begin{equation*}
  \fpsf(s)
  =
  s^{\lambda}
  (\ln s)^{\sigma}
  (1+\smo(1))
  \,,
  \qquad
  s\to+\infty
  \,,
\end{equation*}
for
\begin{gather*}
  \lambda
  =
  p-\alpha
  +
  \big(
  1
  +
  \beta(\SpDim-1)
  -
  \alpha
  \big)
  (p+m-3)
  \,,
  \\
  \sigma
  =
  \nu
  (\SpDim-1)
  (p+m-3)
  +
  \mu
  (p+m-2)
  \,,
\end{gather*}
where $\lambda>0$ owing to our assumptions.
Then
\begin{equation*}
  \fpsf^{(-1)}(v)
  =
  \lambda^{\frac{\sigma}{\lambda}}
  v^{\frac{1}{\lambda}}
  (\ln v)^{-\frac{\sigma}{\lambda}}
  (1+\smo(1))
  \,,
  \qquad
  v\to+\infty
  \,.
\end{equation*}
Hence Theorem~\ref{t:sup} applies and yields an asymptotic decay rate of $t^{-\delta_{1}}(\ln t)^{\delta_{2}}$ with
\begin{equation*}
  \delta_{1}
  =
  \frac{
    \beta(\SpDim-1)+1-\alpha
  }{
    \lambda
  }
  \,,
  \qquad
  \delta_{2}
  =
  \sigma
  \delta_{1}
  -
  \mu
  -
  \nu(\SpDim-1)
  \,.
\end{equation*}
Also the estimate of finite speed of propagation in Theorem~\ref{t:fsp} is in force for large times, provided $p+m-3>0$, and \eqref{eq:exa_nn} is satisfied.

Concerning the universal bound result in Theorem~\ref{t:unb}, the necessary assumptions in this example are fulfilled if $\alpha>p$, $p+m-3>0$.

Finally, the interface blow up phenomenon of Theorem~\ref{t:ibl} takes place if $p+m-3>0$ and
\begin{equation}
  \label{eq:exa_mm}
  \alpha
  >
  \alpha^{*}
  \,,
\end{equation}
which guarantees \eqref{eq:ibl_n}. Note that $\alpha^{*}>p$.

If we take formally above $\beta=1$, $\mu=\nu=0$, we recover the results known in the setting of the Euclidean space (\cite{Tedeev:2007}).
Furthermore, the example is still admissible, under the same assumptions, if $f$ and $\dnf$ are modified by multiplying them by a factor $1+H(\tau)$ where $H$ is a sufficiently regular and fast decaying function as $\tau\to+\infty$.

\subsection{Plan of the paper.}
\label{s:plan}
We prove in Section~\ref{s:aux} several auxiliary inequalities. In
Section~\ref{s:sup} we prove Theorem~\ref{t:sup} in the degenerate
case, while the proof in the singular case, being a minor variant of the previous
one, is dealt with in the short Section~\ref{s:sin}. The finite speed
of propagation property is proved in Section~\ref{s:fsp}. The
supercritical universal sup bound is proved in Section~\ref{s:unb},
while finally the interface blow up is treated in Section~\ref{s:ibl}.

\section{Auxiliary results}
\label{s:aux}

We begin by stating the following trivial result.
\begin{lemma}
  \label{l:aux_elem}
  Let $f\in C([0,+\infty))$ be a nonnegative function such that for given $r_{0}$, $c_{0}$, $c_{1}$, $\beta>0$, $f$ is nondecreasing in $(r_{0},+\infty)$ and satisfies 
  \begin{equation}
    \label{eq:aux_elem_n}
    c_{0}s^{\beta}
    \le
    f(s)
    \le
    c_{1}
    s^{\beta}
    \,,
    \qquad
    0\le s\le r_{0}
    \,.
  \end{equation}
  Then for all $s>r>0$
  \begin{equation}
    \label{eq:aux_elem_nn}
    f(r)
    \le
    c_{1}c_{0}^{-1}
    f(s)
    \,.
  \end{equation}
\end{lemma}

A first consequence of Lemma~\ref{l:aux_elem} applied to
$f(s)=\dnf(s)s^{\alpha_{2}}$, and of \eqref{eq:dnf_inc}, is the
following one: for all $s>r>0$
\begin{gather}
  \label{eq:aux_dnf_ge1}
  \dnf(r)
  r^{\alpha_{2}}
  \le
  C
  \dnf( s)
  s^{\alpha_{2}}
  \,,
\end{gather}
where $C=\dnf(0)/\dnf(1)\ge 1$.

\begin{lemma}
  \label{l:aux_vold}
  Assume \eqref{eq:M_phyp}; then for the same constant $c$ as in \eqref{eq:M_phyp}
  \begin{equation}
    \label{eq:aux_vold_n}
    \vol(s)
    \ge
    c
    \Big(\frac{s}{r}\Big)^{p}
    \vol(r)
    \,,
    \qquad
    s>r>0
    \,.
  \end{equation}
  Assume further \eqref{eq:dnf_inc}; then
  a constant $\tilde c>0$ exists such that
  \begin{equation}
    \label{eq:aux_vold_nn}
    \vold(s)
    \ge
    \tilde c
    \Big(\frac{s}{r}\Big)^{p-\alpha_{2}}
    \vold(r)
    \,,
    \qquad
    s>r>0
    \,.
  \end{equation}
  As a consequence
  \begin{equation}
    \label{eq:aux_vold_nnn}
    \dlov(a)
    \le
    \tilde c^{-\frac{1}{p-\alpha_{2}}}
    \Big(\frac{a}{b}\Big)^{\frac{1}{p-\alpha_{2}}}
    \dlov(b)
    \,,
    \qquad
    a>b>0
    \,.
  \end{equation}
\end{lemma}

\begin{proof}
  The inequality \eqref{eq:aux_vold_n} follows from \eqref{eq:M_phyp}, since on taking there $k=\vol(s)$ we get
  \begin{equation}
    \label{eq:aux_vold_i}
    c^{-1}
    \frac{\vol(s)}{s^{p}}
    \ge
    \int_{0}^{\vol(s)}
    \frac{\di t}{\vol^{(-1)}(t)^{p}}
    \ge
    \int_{0}^{\vol(r)}
    \frac{\di t}{\vol^{(-1)}(t)^{p}}
    \ge
    \frac{\vol(r)}{r^{p}}
    \,,
  \end{equation}
  when we also exploit the fact that the integrand is decreasing.

  Next we compute appealing to \eqref{eq:aux_dnf_ge1} and to \eqref{eq:aux_vold_n}
  \begin{equation}
    \label{eq:aux_vold_ii}
    \vold(s)
    =
    \vol(s)\dnf(s)s^{\alpha_{2}}s^{-\alpha_{2}}
    \ge
    c
    \Big(\frac{s}{r}\Big)^{p}
    \vol(r)
    C^{-1}
    \dnf(r)
    r^{\alpha_{2}}
    s^{-\alpha_{2}}
    \,,
  \end{equation}
  whence \eqref{eq:aux_vold_nn}. Finally \eqref{eq:aux_vold_nnn} is a simple consequence of \eqref{eq:aux_vold_nn}.
\end{proof}

\begin{lemma}
  \label{l:aux_vol}
  Under the assumptions \eqref{eq:M_iso}, \eqref{eq:M_grow}, \eqref{eq:M_growup}
  we have
  \begin{gather}
    \label{eq:aux_vol_n}
    \gamma^{-1}
    \lambda
    \vol(R)
    \le
    \vol(\lambda R)
    \le
    \gamma
    \lambda^{\SpDim}
    \vol(R)
    \,,
    \qquad
    \lambda \ge 1
    \,,
    \\
    \label{eq:aux_vol_nn}
    \gamma^{-1}
    \lambda^{\SpDim}
    \vol(R)
    \le
    \vol(\lambda R)
    \le
    \gamma
    \lambda
    \vol(R)
    \,,
    \qquad
    0< \lambda \le 1
    \,.
  \end{gather}
\end{lemma}

\begin{proof}
  Define $F(v)=v/\iso(v)$, for $v>0$.
  Owing to our assumption \eqref{eq:M_iso} we have for $k\ge1$
  \begin{equation}
    \label{eq:aux_vol_i}
    F(k v)
    =
    \frac{kv}{\iso(kv)}
    =
    \frac{kv}{(kv)^{\isoexp}}
    \frac{(kv)^{\isoexp}}{\iso(kv)}
    \ge
    k^{\frac{1}{\SpDim}}
    F(v)
    \,.
  \end{equation}
  Similarly for $0<k\le 1$ and $v>0$
  \begin{equation}
    \label{eq:aux_vol_ii}
    F(k v)
    \le
    k^{\frac{1}{\SpDim}}
    F(v)
    \,.
  \end{equation}
  It follows that $F$ is increasing and
  \begin{equation*}
    F(v)
    \to
    0
    \,,
    \quad
    v\to 0+
    \,,
    \qquad
    F(v)
    \to
    +\infty
    \,,
    \quad
    v\to+\infty
    \,.
  \end{equation*}
  In addition, from the monotonicity of $\iso$ we get
  \begin{equation}
    \label{eq:aux_vol_v}
    F(\lambda v)
    \le
    \lambda
    F(v)
    \quad
    \lambda\ge 1
    \,;
    \qquad
    F(\lambda v)
    \ge
    \lambda
    F(v)
    \,,
    \quad
    0<\lambda\le 1
    \,.
  \end{equation}
  Then for given $\lambda$, $R>0$ we let $v=F^{(-1)}(R)$, $k=\lambda^{\SpDim}$ and infer from \eqref{eq:aux_vol_i}--\eqref{eq:aux_vol_v},
  \begin{gather}
    \label{eq:aux_vol_iii}
    \lambda
    F^{(-1)}(R)
    \le
    F^{(-1)}(\lambda R)
    \le
    \lambda^{\SpDim}
    F^{(-1)}(R)
    \,,
    \qquad
    \lambda\ge 1
    \,;
    \\
    \label{eq:aux_vol_iv}
    \lambda^{\SpDim}
    F^{(-1)}(R)
    \le
    F^{(-1)}(\lambda R)
    \le
    \lambda
    F^{(-1)}(R)
    \,,
    \qquad
    0<\lambda\le 1
    \,.
  \end{gather}
  Next we invoke assumptions \eqref{eq:M_grow}, \eqref{eq:M_growup} to infer
  \begin{multline}
    \label{eq:aux_vol_vi}
    c^{-1}
    R
    =
    c^{-1}
    \int_{0}^{\vol(R)}
    \Big(
    \der{\vol}{R}(s)
    \Big)^{-1}
    \di s
    \ge
    \int_{0}^{\vol(R)}
    \frac{\di s}{\iso(s)}
    \\
    \ge
    c
    \int_{0}^{\vol(R)}
    \Big(
    \der{\vol}{R}(s)
    \Big)^{-1}
    \di s
    =
    c
    R
    \,.
  \end{multline}
  However, again from \eqref{eq:M_iso},
  \begin{equation}
    \label{eq:aux_vol_vii}
    \int_{0}^{\vol(R)}
    \frac{\di s}{\iso(s)}
    =
    \int_{0}^{\vol(R)}
    s^{-\isoexp}
    \frac{s^{\isoexp}}{\iso(s)}
    \di s
    \le
    \SpDim
    \frac{\vol(R)}{\iso(\vol(R))}
    =
    \SpDim
    F(\vol(R))
    \,,
  \end{equation}
  while invoking again the monotonicity of $\iso$
  \begin{multline}
    \label{eq:aux_vol_viii}
    \int_{0}^{\vol(R)}
    \frac{\di s}{\iso(s)}
    \ge
    \int_{\frac{\vol(R)}{2}}^{\vol(R)}
    \frac{\di s}{\iso(s)}
    =
    \int_{\frac{\vol(R)}{2}}^{\vol(R)}
    s^{-\isoexp}
    \frac{s^{\isoexp}}{\iso(s)}
    \di s
    \\
    \ge
    \frac{\SpDim}{2}
    (2^{\frac{1}{\SpDim}}-1)
    \frac{\vol(R)}{\iso(\vol(R)/2)}
    \ge
    \frac{\SpDim}{2}
    (2^{\frac{1}{\SpDim}}-1)
    F(
    \vol(R)
    )
    \,.
  \end{multline}
  On combining \eqref{eq:aux_vol_vi}--\eqref{eq:aux_vol_viii}, we get
  \begin{equation}
    \label{eq:aux_vol_j}
    \gamma_{\SpDim,c}^{-1}
    R
    \le
    F(\vol(R))
    \le
    \gamma_{\SpDim,c}
    R
    \,,
    \qquad
    R>0
    \,,
  \end{equation}
  for a suitable $\gamma_{\SpDim,c}>1$. Thus owing to \eqref{eq:aux_vol_iii}--\eqref{eq:aux_vol_iv} we have
  \begin{equation}
    \label{eq:aux_vol_jj}
    \gamma^{-1}
    F^{(-1)}
    (R)
    \le
    \vol(R)
    \le
    \gamma
    F^{(-1)}
    (R)
    \,,
  \end{equation}
  for a suitable $\gamma>1$.

  Our claims \eqref{eq:aux_vol_n}, \eqref{eq:aux_vol_nn} finally follow from \eqref{eq:aux_vol_jj} and again from \eqref{eq:aux_vol_iii}--\eqref{eq:aux_vol_iv}.
\end{proof}

\begin{lemma}
  \label{l:aux_density}
  Under assumptions \eqref{eq:dnf_inc}, \eqref{eq:M_phyp}
  we have for $R>0$
  \begin{equation}
    \label{eq:aux_density_n}
    \gamma^{-1}
    \int_{B_{R}}
    \dnf(d(x))
    \di \msr
    \le
    \dnf(R)
    \vol(R)
    \le
    \int_{B_{R}}
    \dnf(d(x))
    \di \msr
    \,.
  \end{equation}
\end{lemma}

\begin{proof}
  In fact simply by monotonicity of $\dnf$ we have that
  \begin{equation}
    \label{eq:aux_density_i}
    \int_{B_{R}}
    \dnf(d(x))
    \di \msr
    \ge
    \dnf(R)
    \vol(R)
    \,.
  \end{equation}
  Then we calculate, exploiting \eqref{eq:aux_dnf_ge1},
  \begin{multline}
    \label{eq:aux_density_ii}
    \int_{B_{R}}
    \dnf(d(x))
    \di \msr
    =
    \int_{0}^{R}
    \dnf(\tau)
    \der{\vol}{\tau}(\tau)
    \di\tau
    \le
    C
    \int_{0}^{R}
    \big(
    \dnf(R)
    R^{\alpha_{2}}
    \big)
    \tau^{-\alpha_{2}}
    \der{\vol}{\tau}(\tau)
    \di\tau
    \\
    =
    C
    \dnf(R)
    R^{\alpha_{2}}
    \int_{0}^{R}
    \tau^{-\alpha_{2}}
    \der{\vol}{\tau}(\tau)
    \di\tau
    \,.
  \end{multline}
  The integral in \eqref{eq:aux_density_ii} is handled by means of the change of variable $s=\vol(\tau)$, yielding
  \begin{equation}
    \label{eq:aux_density_iii}
    \begin{split}
      \int_{0}^{R}
      \tau^{-\alpha_{2}}
      \der{\vol}{\tau}(\tau)
      \di\tau
      &=
      \int_{0}^{\vol(R)}
      \frac{\di s}{\vol^{(-1)}(s)^{\alpha_{2}}}
      =
      \int_{0}^{\vol(R)}
      \frac{\vol^{(-1)}(s)^{p-\alpha_{2}}}{\vol^{(-1)}(s)^{p}}
      \di s
      \\
      &\le
      \vol^{(-1)}(\vol(R))^{p-\alpha_{2}}
      \int_{0}^{\vol(R)}
      \frac{\di s}{\vol^{(-1)}(s)^{p}}
      \\
      &
      \le
      c^{-1}
      R^{p-\alpha_{2}}
      \frac{\vol(R)}{\vol^{(-1)}(\vol(R))^{p}}
      =
      c^{-1}
      R^{-\alpha_{2}}
      \vol(R)
      \,.
    \end{split}
  \end{equation}
  Note that we used also \eqref{eq:M_phyp}.
  
  Collecting \eqref{eq:aux_density_ii}, \eqref{eq:aux_density_iii} we obtain the claim.
\end{proof}

\begin{lemma}
  \label{l:aux_function}
  Under assumptions \eqref{eq:M_phyp}, \eqref{eq:M_inc}, \eqref{eq:dnf_dec}, \eqref{eq:dnf_inc}, we have for all $s>r>0$
  \begin{equation}
    \label{eq:aux_function_n}
    W(r)
    :=
    \dnf(\lov_{\dnf}(r))
    \lov_{\dnf}(r)^{p}
    r^{-\frac{p-\alpha_{2}}{\SpDim-\alpha_{1}}}
    \le
    \gamma
    W(s)
    \,,
  \end{equation}
  for a suitable $\gamma>1$. In addition for $\lambda \ge 1$, $r>0$
  \begin{equation}
    \label{eq:aux_function_nn}
    W(\lambda r)
    \le
    \gamma_{1}
    \lambda^{d}
    W(r)
    \,,
  \end{equation}
  where actually $\gamma_{1}=\tilde c^{-(p-\alpha_{1})/(p-\alpha_{2})}$ for $\tilde c$ as in \eqref{eq:aux_vold_nn}, and
  \begin{equation*}
    d
    =
    \frac{p-\alpha_{1}}{p-\alpha_{2}}
    -
    \frac{p-\alpha_{2}}{\SpDim-\alpha_{1}}
    >0
    \,.
  \end{equation*}
\end{lemma}

\begin{proof}
  We appeal to Lemma~\ref{l:aux_elem} with $f=W$, $r_{0}=\vold(1)>0$.

  Let us begin by checking that $W$ is nondecreasing in $(\vold(1),+\infty)$; in this case $\dlov(s)>1$.
  Write
  \begin{equation*}
    W(s)
    =
    [
    \dnf(\lov_{\dnf}(s))
    \lov_{\dnf}(s)^{\alpha_{2}}
    ]
    \,
    [
    \lov_{\dnf}(s)
    s^{-\frac{1}{\SpDim-\alpha_{1}}}
    ]^{p-\alpha_{2}}
    \,.
  \end{equation*}
  The first factor is nondecreasing by assumption \eqref{eq:dnf_inc}. As to the second factor, set $R=\lov_{\dnf}(s)$. Then the quantity in square brackets in such a factor equals
  \begin{equation*}
    R
    [
    \dnf(R)
    \vol(R)
    ]^{-\frac{1}{\SpDim-\alpha_{1}}}
    =
    [
    \dnf(R)
    R^{\alpha_{1}}
    ]^{-\frac{1}{\SpDim-\alpha_{1}}}
    \,
    \Big[
    \frac{R^{\SpDim}}{\vol(R)}
    \Big]^{\frac{1}{\SpDim-\alpha_{1}}}
    \,.
  \end{equation*}
  Here, the first factor is nondecreasing by assumption \eqref{eq:dnf_dec}; the second one is nondecreasing by assumption \eqref{eq:M_inc}.

  Next, we note that clearly two constants $C_{1}>C_{0}>0$ exist such that
  \begin{equation}
    \label{eq:aux_function_i}
    C_{0}
    s^{\frac{1}{\SpDim}}
    \le
    \dlov(s)
    \le
    C_{1}
    s^{\frac{1}{\SpDim}}
    \,,
    \qquad
    \text{for $\dlov(s)\le 1$.}
  \end{equation}
  Thus
  \begin{equation}
    \label{eq:aux_function_iii}
    \dnf(1)
    C_{0}
    s^{\frac{p}{\SpDim}-\frac{p-\alpha_{2}}{\SpDim-\alpha_{1}}}
    \le
    W(s)
    \le
    \dnf(0)
    C_{1}
    s^{\frac{p}{\SpDim}-\frac{p-\alpha_{2}}{\SpDim-\alpha_{1}}}
    \,.
  \end{equation}
  It is easy to check that
  \begin{equation*}
    \beta
    =
    \frac{p}{\SpDim}-\frac{p-\alpha_{2}}{\SpDim-\alpha_{1}}
    =
    \frac{
      \SpDim\alpha_{2}
      -
      p\alpha_{1}
    }{
      \SpDim(\SpDim-\alpha_{1})
    }
    >0
    \,,
  \end{equation*}
  owing to our assumptions $\SpDim>p>\alpha_{2}>\alpha_{1}$. Thus for such a $\beta$ and $c_{0}$, $c_{1}$ given in \eqref{eq:aux_function_iii} we have proved \eqref{eq:aux_elem_n}. Our claim \eqref{eq:aux_function_n} follows.

  Our second claim \eqref{eq:aux_function_nn} is a direct consequence of \eqref{eq:dnf_dec} and of \eqref{eq:aux_vold_nnn}.
\end{proof}

\begin{lemma}
  \label{l:sob}
  (\cite{Andreucci:Tedeev:2019,Minerbe:2009})
  Assume that $1<p<\SpDim$, and that \eqref{eq:M_iso}, \eqref{eq:M_grow}, \eqref{eq:M_growup}, \eqref{eq:M_phyp} hold true. Then
  for all $\unk\in W^{1,p}(M)$ we have
  \begin{equation*}
    \Big(
    \int_{M}
    \abs{\unk}^{p^{*}}
    \ipf\big(\vol(d(x))\big)^{-p^{*}}
    \di \msr
    \Big)^{\frac{\SpDim-p}{N}}
    \le
    C
    \int_{M}
    \abs{\grad \unk}^{p}
    \di \msr
    \,,
  \end{equation*}
  where $p^{*}=p\SpDim/(\SpDim-p)$ and $C>0$ is a suitable constant.
\end{lemma}

\begin{proof}
  Our assumptions match the ones made in \cite{Andreucci:Tedeev:2019}, when we show that the inequality
  \begin{equation}
    \label{eq:sob_i}
    \int_{0}^{k}
    s^{-p}
    h(s)^{p}
    \di s
    \le
    \gamma
    k^{-p+1}
    h(k)
    \,,
    \qquad
    k>0
    \,,
  \end{equation}
  follows from our assumptions. Indeed, setting $R=\vol^{(-1)}(s)$, we have, with the notation of the proof of Lemma~\ref{l:aux_vol},
  \begin{equation}
    \label{eq:sob_ii}
    s
    h(s)^{-1}
    =
    \vol(R)
    h(\vol(R))^{-1}
    =
    F(\vol(R))
    \,,
  \end{equation}
  and therefore by \eqref{eq:aux_vol_jj}
  \begin{equation}
    \label{eq:sob_iii}
    \gamma_{\SpDim,c}^{-1}
    \vol^{(-1)}(s)
    \le
    s
    h(s)^{-1}
    \le
    \gamma_{\SpDim,c}
    \vol^{(-1)}(s)
    \,,
    \qquad
    s>0
    \,.
  \end{equation}
  Then \eqref{eq:sob_i} follows from \eqref{eq:sob_iii} and \eqref{eq:M_phyp}.
\end{proof}

The weighted Sobolev inequality in \cite{Minerbe:2009} was proved under the assumption
that the Ricci curvature is nonnegative and the volume growth is such to guarantee
the hyperbolicity of the manifold. Instead in \cite{Andreucci:Tedeev:2019}
we applied, to the same end, a symmetrization technique relying on an
isoperimetrical inequality, as well as an assumption of $p$-hyperbolicity of $M$.

Next we prove the following Hardy inequality, needed below; see also the survey \cite{DAmbrosio:Dipierro:2014} on this subject.

\begin{theorem}[Hardy inequality]
  \label{t:hardy}
  Assume \eqref{eq:M_iso}, \eqref{eq:M_phyp}, \eqref{eq:M_isoup}.
  For any $\unk\in W^{1,p}(M)$ we have
  \begin{equation}
    \label{eq:hardy_n}
    \int_{M}
    \frac{\abs{\unk}^{p}}{d(x)^{p}}
    \di\msr
    \le
    \gamma(\SpDim,p)
    \int_{M}
    \abs{\grad \unk}^{p}
    \di\msr
    \,.
  \end{equation}
\end{theorem}

\begin{proof}
  We may assume $\unk\ge 0$.
  Introduce for $\lambda\ge 0$ the standard rearrangement function
  \begin{equation}
    \label{eq:rearr}
    \unk^{*}(s)
    =
    \inf\{\lambda \mid \msr_{\lambda}<s\}
    \,,
    \quad
    \msr_{\lambda}
    =
    \abs{\{x\in M \mid \abs{\unk(x)}> \lambda\}}
    \,.
  \end{equation}
  We have
  \begin{equation}
    \label{eq:hardy_i}
    \int_{M}
    \frac{\unk^{p}}{d(x)^{p}}
    \di\msr
    \le
    \int_{0}^{+\infty}
    \unk^{*}(s)^{p}
    [d(\cdot)^{-p}]^{*}(s)
    \di s
    \,.
  \end{equation}
  On the other hand
  \begin{equation*}
    \abs{\{d(x)^{-p}>\lambda\}}
    =
    \abs{B_{\lambda^{-\frac{1}{p}}}}
    =
    \vol(\lambda^{-\frac{1}{p}})
    \,.
  \end{equation*}
  Therefore \eqref{eq:hardy_i} gives on integrating by parts
  \begin{equation}
    \label{eq:hardy_ii}
    \int_{M}
    \frac{\unk^{p}}{d(x)^{p}}
    \di\msr
    \le
    \int_{0}^{+\infty}
    \frac{\unk^{*}(s)^{p}}{\vol^{(-1)}(s)^{p}}
    \di s
    =
    p
    \int_{0}^{+\infty}
    \unk^{*}(s)^{p-1}
    [-\unk_{s}^{*}(s)]
    \int_{0}^{s}
    \frac{\di\tau}{\vol^{(-1)}(\tau)^{p}}
    \di s
    \,.
  \end{equation}
  Next we apply our assumption \eqref{eq:M_phyp} in \eqref{eq:hardy_ii} and
  after applying H\"older inequality we arrive at
  \begin{equation}
    \label{eq:hardy_iii}
    \int_{0}^{+\infty}
    \frac{\unk^{*}(s)^{p}}{\vol^{(-1)}(s)^{p}}
    \di s
    \le
    \gamma
    \Big(
    \int_{0}^{+\infty}
    \frac{\unk^{*}(s)^{p}}{\vol^{(-1)}(s)^{p}}
    \di s
    \Big)^{\frac{p-1}{p}}
    \Big(
    \int_{0}^{+\infty}
    [-\unk^{*}_{s}(s)]^{p}
    \frac{s^{p}}{\vol^{(-1)}(s)^{p}}
    \di s
    \Big)^{\frac{1}{p}}
    \,.
  \end{equation}
  This immediately yields when we invoke \eqref{eq:M_isoup}
  \begin{multline}
    \label{eq:hardy_iv}
    \int_{0}^{+\infty}
    \frac{\unk^{*}(s)^{p}}{\vol^{(-1)}(s)^{p}}
    \di s
    \le
    \gamma
    \int_{0}^{+\infty}
    [-\unk^{*}_{s}(s)]^{p}
    \frac{s^{p}}{\vol^{(-1)}(s)^{p}}
    \di s
    \\
    \le
    \gamma
    \int_{0}^{+\infty}
    [-\unk^{*}_{s}(s)]^{p}
    h(s)^{p}
    \di s
    \le
    \gamma
    \int_{M}
    \abs{\grad\unk}^{p}
    \di\msr
    \,,
  \end{multline}
  that is \eqref{eq:hardy_n}, by Polya-Szego principle (see \cite{Andreucci:Tedeev:2019}).
\end{proof}

We state first an estimate where the density function $\dnf$ does not appear.

\begin{lemma}
  \label{l:emb_old}
  Let $\unk\in W^{1,p}(M)$, $0<r<q\le \SpDim p/(\SpDim-p)$. Then
  \begin{equation}
    \label{eq:emb_old_n}
    \int_{M}
    \abs{\unk}^{q}
    \di\msr
    \le
    \gamma
    \ipf(S_{q})^{q}
    S_{q}^{1+\frac{q}{\SpDim}-\frac{q}{p}}
    \norma{\grad \unk}{L^{p}(M)}^{q}
    \,,
  \end{equation}
  where
  \begin{equation}
    \label{eq:emb_old_nn}
    S_{q}
    =
    \big(
    \int_{M}
    \abs{\unk}^{r}
    \di\msr
    \Big)^{\frac{q}{q-r}}
    \Big(
    \int_{M}
    \abs{\unk}^{q}
    \di\msr
    \Big)^{-\frac{r}{q-r}}
    \,.
  \end{equation}
\end{lemma}

\begin{proof}
  We confine ourselves to the case $q\le p$, which is the one of our interest here. The case $q>p$ can be proved reasoning as in \cite{Andreucci:Tedeev:2000}.
  \\
  Introduce the standard rearrangement function as in \eqref{eq:rearr}.
  Then write for convenience of notation
  \begin{equation*}
    P_{s}
    =
    \int_{M}
    \abs{\unk(x)}^{s}
    \di\msr
    \,,
    \qquad
    s>0
    \,.
  \end{equation*}
  We have for a $k>0$ to be selected presently
  \begin{multline}
    \label{eq:emb_old_i}
    P_{q}
    =
    \int_{0}^{\msr_{0}}
    \unk^{*}(s)^{q}
    \di s
    \le
    \gamma(q)
    \int_{0}^{\msr_{k}}
    (\unk^{*}(s)-k)^{q}
    \di s
    +
    \gamma(q)
    k^{q}
    \msr_{k}
    +
    \int_{\msr_{k}}^{\msr_{0}}
    \unk^{*}(s)^{q}
    \di s
    \\
    =:
    I_{1}+I_{2}+I_{3}
    \,.
  \end{multline}
  Next we invoke Chebychev inequality
  \begin{equation*}
    k^{r}
    \msr_{k}
    \le
    P_{r}
    \,,
  \end{equation*}
  to bound
  \begin{equation}
    \label{eq:emb_old_ii}
    I_{2}+I_{3}
    \le
    \gamma
    \msr_{k}^{1-\frac{q}{r}}
    P_{r}^{\frac{q}{r}}
    +
    k^{q-r}
    \int_{\msr_{k}}^{\msr_{0}}
    \unk^{*}(s)^{r}
    \di s
    \le
    \gamma
    \msr_{k}^{1-\frac{q}{r}}
    P_{r}^{\frac{q}{r}}
    =
    \frac{1}{2}
    P_{q}
    \,.
  \end{equation}
  The last equality in \eqref{eq:emb_old_ii} is our choice of
  $k$, which amounts to $\msr_{k}=\gamma S_{q}$.  
  Note that we may assume $\msr_{0}$ as large as necessary, by
  approximating $\unk$ while keeping all the involved integral quantities
  stable. Thus we can safely assume that such a value of $k$
  exists. Hence we absorb $I_{2}+I_{3}$ on the left hand side of
  \eqref{eq:emb_old_i}. We then reason as in \cite{Talenti:1976} to obtain
  \begin{multline}
    \label{eq:emb_old_iii}
    P_{q}
    \le
    \gamma
    \int_{0}^{\msr_{k}}
    (\unk^{*}(s)-k)^{q}
    \di s
    \le
    \gamma
    \msr_{k}^{1-\frac{q}{p}}
    \Big(
    \int_{0}^{\msr_{k}}
    (\unk^{*}(s)-k)^{p}
    \di s
    \Big)^{\frac{q}{p}}
    \\
    \le
    \gamma
    \msr_{k}^{1-\frac{q}{p}}
    \Big(
    \int_{0}^{\msr_{k}}
    [-\unk^{*}_{s}(s)]^{p}
    h(s)^{p}
    [s
    h(s)^{-1}]^{p}
    \di s
    \Big)^{\frac{q}{p}}
    \\
    \le
    \gamma
    \msr_{k}^{1-\frac{q}{p}}
    [\msr_{k} h(\msr_{k})^{-1}]^{q}
    \Big(
    \int_{M}
    \abs{\grad \unk}^{p}
    \di\msr
    \Big)^{\frac{q}{p}}
    \,.
  \end{multline}
  We have exploited here the fact that $t\mapsto t h(t)^{-1}$ is
  increasing as it follows from our assumption that $\ipf$ is nondecreasing.

  Finally \eqref{eq:emb_old_n} follows from \eqref{eq:emb_old_iii} and
  from our choice $\msr_{k}=\gamma S_{q}$.
\end{proof}

\begin{corollary}
  \label{co:emb_old_p}
  Let $\unk\in W^{1,p}(M)$ and $0<r<p$. Then
  \begin{multline}
    \label{eq:emb_old_p_n}
    \int_{M}
    \abs{\unk}^{p}
    \di\msr
    \le
    \gamma
    \ipf(\msr(\supp \unk))^{\frac{p\SpDim(p-r)}{\SpDim(p-r)+r p}}
    \Big(
    \int_{M}
    \abs{\unk}^{r}
    \di\msr
    \Big)^{\frac{p^{2}}{\SpDim(p-r)+r p}}
    \\
    \times
    \Big(
    \int_{M}
    \abs{\grad \unk}^{p}
    \di\msr
    \Big)^{\frac{\SpDim(p-r)}{\SpDim(p-r)+r p}}
    \,.
  \end{multline}
\end{corollary}

\begin{proof}
  We select $q=p$ in Lemma~\ref{l:emb_old}. The statement follows
  from an elementary computation, when we also bound by means of
  H\"older's inequality
  \begin{equation}
    \label{eq:emb_old_p_i}
    S_{q}
    \le
    \Big[
    \msr(\supp \unk)^{1-\frac{r}{p}}
    \Big(
    \int_{M}
    \abs{\unk}^{p}
    \di\msr
    \Big)^{\frac{r}{p}}
    \Big]^{\frac{p}{p-r}}
    \Big(
    \int_{M}
    \abs{\unk}^{p}
    \di\msr
    \Big)^{-\frac{r}{p-r}}
    =
    \msr(\supp \unk)
    \,.
  \end{equation}
\end{proof}

Next we state some weighted estimates where the estimated norms involve the weight $\dnf$.
We denote in the following for $k> 0$
\begin{equation}
  \label{eq:measd}
  \msd(A)
  =
  \int_{A}
  \dnf(d(x))
  \di \msr
  \,,
  \qquad
  \measd(k)
  =
  \msd(\{u>k\})
  \,.
\end{equation}

\begin{lemma}[Faber-Krahn type estimate]
  \label{l:fk}
  Assume that $1<p<\SpDim$, and that the assumptions of Lemma~\ref{l:sob} and of Theorem~\ref{t:hardy} hold true. We further assume  \eqref{eq:dnf_inc}, \eqref{eq:close}.
  \\
  Then
  for all $\unk\in W^{1,p}(M)$, $k>0$ we have
  \begin{equation}
    \label{eq:fk_n}
    \int_{\{u>k\}}
    \dnf(d(x))
    (u-k)^{p}
    \di \msr
    \le
    \gamma
    \dnf(\dlov(\measd(k)))
    \dlov(\measd(k))^{p}
    \int_{\{u>k\}}
    \abs{\grad u}^{p}
    \di \msr
    \,.
  \end{equation}
\end{lemma}

\begin{proof}
  We split the integral between $B(R)$ and its complement; let $A_{k}=\{u>k\}$. We first have from Lemma~\ref{l:sob},
  \begin{equation}
    \label{eq:fk_i}
    \begin{split}
      &\int_{A_{k}\setminus B_{R}}
      \dnf(d(x))
      (u-k)^{p}
      \di \msr
      \le
      \Big(
      \int_{A_{k}}
      (u-k)^{p^{*}}
      \ipf\big(\vol(d(x))\big)^{-p^{*}}
      \di \msr
      \Big)^{\frac{\SpDim-p}{\SpDim}}
      \times
      \\
      &\qquad\Big(
      \int_{A_{k}\setminus B_{R}}
      \dnf(d(x))^{\frac{\SpDim}{p}}
      \ipf\big(\vol(d(x))\big)^{\SpDim}
      \di \msr
      \Big)^{\frac{p}{\SpDim}}
      \\
      &\quad
      \le
      \gamma
      \dnf(R)^{\frac{\SpDim-p}{\SpDim}}
      \ipf(\vol(R))^{p}
      \measd(k)^{\frac{p}{\SpDim}}
      \Big(
      \int_{A_{k}}
      \abs{\grad u}^{p}
      \di \msr
      \Big)
      \\
      &\quad
      \le
      \gamma
      \dnf(R)^{\frac{\SpDim-p}{\SpDim}}
      R^{p}
      \vol(R)^{-\frac{p}{\SpDim}}
      \measd(k)^{\frac{p}{\SpDim}}
      \Big(
      \int_{A_{k}}
      \abs{\grad u}^{p}
      \di \msr
      \Big)
      =:
      I_{1}
      \,.
    \end{split}
  \end{equation}
  We used assumption \eqref{eq:close} and the estimate \eqref{eq:aux_vol_j}, at least for $R>s_{0}$ as in \eqref{eq:close}. For $R\le s_{0}$ we simply note that
  \begin{equation*}
    c
    \le
    \dnf(s)
    \ipf(\vol(s))^{p^{*}}
    \le
    c^{-1}
    \,,
    \qquad
    0<s\le s_{0}
    \,,
  \end{equation*}
  for a suitable $0<c<1$, as $\ipf(0+)>0$ due to the locally Euclidean behavior of the Riemannian metric. 

  Next we have by means of Hardy inequality \eqref{eq:hardy_n}, and also taking into account \eqref{eq:aux_dnf_ge1},
  \begin{multline}
    \label{eq:fk_ii}
    \int_{A_{k}\cap B(R)}
    \dnf(d(x))
    (u-k)^{p}
    \di \msr
    =
    \int_{A_{k}\cap B(R)}
    \dnf(d(x))
    d(x)^{p}
    d(x)^{-p}
    (u-k)^{p}
    \di \msr
    \\
    \le
    \gamma
    \dnf(R)
    R^{p}
    \int_{A_{k}}
    \abs{\grad u}^{p}
    \di \msr
    =:
    I_{2}
    \,.
  \end{multline}
  Finally we select $R$ so that, essentially, $I_{1}$ and $I_{2}$ contribute the same quantity, i.e.,
  \begin{equation}
    \label{eq:fk_iii}
    \dnf(R)^{\frac{\SpDim-p}{\SpDim}}
    R^{p}
    \vol(R)^{-\frac{p}{\SpDim}}
    \measd(k)^{\frac{p}{\SpDim}}
    =
    \dnf(R)
    R^{p}
    \,,  
  \end{equation}
  or as one can immediately see
  \begin{equation}
    \label{eq:fk_iv}
    \vold(R)
    =
    \measd(k)
    \,.
  \end{equation}
  The claim follows.
\end{proof}

\begin{lemma}
  \label{l:fks}
  Assume that $1<p<\SpDim$, and that the assumptions of Lemma~\ref{l:sob}, Theorem~\ref{t:hardy} and \eqref{eq:close} hold true. Then
  for all $\unk\in W^{1,p}(M)$ and $k>0$ we have
  \begin{multline}
    \label{eq:fks_n}
    \int_{\{u>k\}}
    \dnf(d(x))
    (u-k)^{s}
    \di \msr
    \le
    \\
    \gamma
    [
    \dnf(\dlov(\measd(k)))
    \dlov(\measd(k))^{p}
    ]^{\frac{s}{p}}
    \measd(k)^{1-\frac{s}{p}}
    \Big(
    \int_{\{u>k\}}
    \abs{\grad u}^{p}
    \di \msr
    \Big)^{\frac{s}{p}}
    \,,
  \end{multline}
  provided $p<s<p^{*}$ and that, in addition, for a given $C>0$
  \begin{equation}
    \label{eq:fks_nn}
    \dnf(R)
    R^{\SpDim-\frac{s(\SpDim-p)}{p}}
    \le
    C
    \dnf(R_{1})
    R_{1}^{\SpDim-\frac{s(\SpDim-p)}{p}}
    \,,
    \qquad
    R_{1}>R>0
    \,.
  \end{equation}
\end{lemma}

\begin{proof}
  As in the proof of Lemma~\ref{l:fk}, we split the integral between
  $B(R)$ and its complement; let $A_{k}=\{u>k\}$ and $\measd(k)$ be defined as in \eqref{eq:measd}.
  \\
  We first have
  from H\"older inequality, Lemma~\ref{l:sob}, and from assumption \eqref{eq:close},
  \begin{equation}
    \label{eq:sin_k}
    \begin{split}
      &\int_{A_{k}\setminus B_{R}}
      \dnf(d(x))
      (u-k)^{s}
      \di\msr
      \le
      \Big(
      \int_{A_{k}\setminus B_{R}}
      \ipf(\vol(d(x)))^{-p^{*}}
      (u-k)^{p^{*}}
      \di\msr
      \Big)^{\frac{s}{p^{*}}}
      \\
      &\quad
      \times
      \Big(
      \int_{A_{k}\setminus B_{R}}
      \dnf(d(x))^{\frac{p^{*}}{p^{*}-s}}
      \ipf(\vol(d(x)))^{\frac{p^{*}s}{p^{*}-s}}
      \di\msr
      \Big)^{1-\frac{s}{p^{*}}}
      \\
      &\quad
      \le
      \gamma
      \dnf(R)^{\frac{s}{p^{*}}}
      \ipf(\vol(R))^{s}
      \measd(k)^{1-\frac{s}{p^{*}}}
      \Big(
      \int_{A_{k}}
      \abs{\grad u}^{p}
      \di\msr
      \Big)^{\frac{s}{p}}
      \\
      &\quad
      \le
      \gamma
      \dnf(R)^{\frac{s}{p^{*}}}
      R^{s}
      \vol(R)^{-\frac{s}{\SpDim}}
      \measd(k)^{1-\frac{s}{p^{*}}}
      \Big(
      \int_{A_{k}}
      \abs{\grad u}^{p}
      \di\msr
      \Big)^{\frac{s}{p}}
      =:
      I_{1}
      \,.
    \end{split}
  \end{equation}
  We have used in last inequality our assumption \eqref{eq:M_isoup}.
  
  Next we apply again H\"older inequality to get
  \begin{multline}
    \label{eq:fks_i}
    \int_{A_{k}\cap B_{R}}
    \dnf(d(x))
    (u-k)^{s}
    \di\msr
    \le
    \Big(
    \int_{A_{k}\cap B_{R}}
    d(x)^{-p}
    (u-k)^{p}
    \di\msr
    \Big)^{\frac{p^{*}-s}{p^{*}-p}}
    \\
    \times
    \Big(
    \int_{A_{k}\cap B_{R}}
    [
    \dnf(d(x))
    d(x)^{\SpDim-(\SpDim-p)\frac{s}{p}}
    ]^{\frac{p^{*}-p}{s-p}}
    (u-k)^{p^{*}}
    \di\msr
    \Big)^{\frac{s-p}{p^{*}-p}}
    \,.
  \end{multline}
  The second factor in \eqref{eq:fks_i} is majorized, owing to our assumption \eqref{eq:fks_nn}, by
  \begin{multline}
    \label{eq:sin_jj}
    \gamma
    \ipf(V(R))^{\frac{\SpDim(s-p)}{p}}
    \dnf(R)
    R^{\SpDim-(\SpDim-p)\frac{s}{p}}
    \Big(
    \int_{A_{k}\cap B_{R}}
    (u-k)^{p^{*}}
    \ipf(\vol(d(x)))^{-p^{*}}
    \di\msr
    \Big)^{\frac{s-p}{p^{*}-p}}
    \\
    \le
    \gamma
    \ipf(V(R))^{\frac{\SpDim(s-p)}{p}}
    \dnf(R)
    R^{\SpDim-(\SpDim-p)\frac{s}{p}}
    \Big(
    \int_{A_{k}}
    \abs{\grad u}^{p}
    \di\msr
    \Big)^{\frac{\SpDim(s-p)}{p^{2}}}
    \,,
  \end{multline}
  when we apply also Lemma~\ref{l:sob}. The first integral in \eqref{eq:fks_i} is bounded with the help of Hardy inequality \eqref{eq:hardy_n}, to obtain finally
  \begin{equation}
    \label{eq:sin_jjj}
    \begin{split}
    &\int_{A_{k}\cap B_{R}}
    \dnf(d(x))
    (u-k)^{s}
    \di\msr
    \\
    &\quad
    \le
    \gamma
    \ipf(V(R))^{\frac{\SpDim(s-p)}{p}}
    \dnf(R)
    R^{\SpDim-(\SpDim-p)\frac{s}{p}}
    \Big(
    \int_{A_{k}}
    \abs{\grad u}^{p}
    \di\msr
    \Big)^{\frac{s}{p}}
    \\
    &\quad
    \le
    \gamma
    R^{\frac{\SpDim(s-p)}{p}}
    \vol(R)^{-\frac{s-p}{p}}
    \dnf(R)
    R^{\SpDim-(\SpDim-p)\frac{s}{p}}
    \Big(
    \int_{A_{k}}
    \abs{\grad u}^{p}
    \di\msr
    \Big)^{\frac{s}{p}}
    =:
    I_{2}
    \,.
    \end{split}
  \end{equation}
  In last inequality we used again \eqref{eq:M_isoup}.
  
  We require that $I_{1}$ and $I_{2}$ contribute essentially the same quantity, that is we define $R$ by means of
  \begin{equation}
    \label{eq:sin_kk}
    \dnf(R)^{\frac{s}{p^{*}}}
    R^{s}
    \vol(R)^{-\frac{s}{\SpDim}}
    \measd(k)^{1-\frac{s}{p^{*}}}
    =
    R^{\frac{\SpDim(s-p)}{p}}
    \vol(R)^{-\frac{s-p}{p}}
    \dnf(R)
    R^{\SpDim-(\SpDim-p)\frac{s}{p}}
    \,.
  \end{equation}
  It is easily seen that this is equivalent to $\vold(R)=\measd(k)$, i.e., $R=\dlov(\measd(k))$. With this choice of $R$ it is trivial to check that, e.g., the right hand side of \eqref{eq:sin_kk} equals $\dnf(R)^{s/p}R^{s}\measd(k)^{1-s/p}$, proving the claim.
\end{proof}

\begin{remark}
  \label{r:fks}
  Concerning assumption \eqref{eq:fks_nn} we note that the exponent of $R$ in it, that is $\SpDim-s(N-p)/p$, tends to $p-$ as $s\to p+$. Hence for suitable $s$ in such a range, \eqref{eq:fks_nn} is a consequence of \eqref{eq:dnf_inc} and of Lemma~\ref{l:aux_elem}.

  Let us also note for later use the following consequence of \eqref{eq:fks_nn}: rewrite the right hand side of \eqref{eq:sin_kk}, that is $\dnf(R)R^{s}\vol(R)^{(p-s)/p}$ as
  \begin{equation*}
    \dnf(R)R^{\SpDim-\frac{s(\SpDim-p)}{p}}
    \Big(
    \frac{R^{N}}{\vol(R)}
    \Big)^{\frac{s-p}{p}}
    \,.
  \end{equation*}
  The last factor above is nondecreasing owing to our assumption \eqref{eq:M_inc}; thus, owing to \eqref{eq:fks_nn} we have
  \begin{equation}
    \label{eq:fks_a}
    \dnf(R)R^{s}\vol(R)^{\frac{p-s}{p}}
    \le
    C
    \dnf(R_{1})R_{1}^{s}\vol(R_{1})^{\frac{p-s}{p}}
    \,,
    \qquad
    0<R<R_{1}
    \,,
  \end{equation}
  for  $C>0$ as in \eqref{eq:fks_nn}.
\end{remark}

In the following we let for the sake of notational simplicity for
$p>r>0$
\begin{equation}
  \label{eq:ES_not}
  E_{r}
  =
  \int_{M}
  \dnf(d(x))
  \abs{u}^{r}
  \di \msr
  \,,
  \qquad
  S
  =
  \frac{E_{r}^{\frac{p}{p-r}}}{E_{p}^{\frac{r}{p-r}}}
  \,.
\end{equation}

\begin{lemma}[Sobolev-Gagliardo-Nirenberg]
  \label{l:sgn}
  Assume that $1<p<\SpDim$, and that the assumptions of Lemma~\ref{l:fk} hold true. Assume further that $0<r<p$. Then we have
  \begin{equation}
    \label{eq:sgn_n}
    E_{p}
    \le
    \gamma
    \dnf(\dlov( S))
    \dlov( S)^{p}
    \int_{M}
    \abs{\grad u}^{p}
    \di \msr
    \,.
  \end{equation}
\end{lemma}

\begin{proof}
  We begin by splitting, for a $k>0$ to be chosen,
  \begin{equation}
    \label{eq:sgn_j}
    E_{p}
    =
    \int_{\{\abs{u}>k\}}
    \dnf(d(x))
    \abs{u}^{p}
    \di \msr
    +
    \int_{\{\abs{u}\le k\}}
    \dnf(d(x))
    \abs{u}^{p}
    \di \msr
    =:
    J_{1}
    +
    J_{2}
    \,.
  \end{equation}
  We first bound, on using \eqref{eq:fk_n} and Chebychev inequality,
  \begin{equation}
    \label{eq:sgn_i}
    \begin{split}
      J_{1}
      &\le
      2^{p-1}
      \int_{\{\abs{u}>k\}}
      \dnf(d(x))
      (\abs{u}-k)^{p}
      \di \msr
      +
      2^{p-1}
      k^{p}
      \measd(k)
      \\
      &\le
      \gamma
      \dnf(\dlov(\measd(k)))
      \dlov(\measd(k))^{p}
      \int_{M}
      \abs{\grad u}^{p}
      \di \msr
      +
      2^{p-1}
      k^{p-r}
      E_{r}
      \,.
    \end{split}
  \end{equation}
  Then we get obviously
  \begin{equation}
    \label{eq:sgn_ii}
    J_{2}
    \le
    k^{p-r}
    E_{r}
    \,.
  \end{equation}
  Collecting \eqref{eq:sgn_j}--\eqref{eq:sgn_ii} we obtain, after a further use of Chebyshev inequality,
  \begin{equation}
    \label{eq:sgn_iii}
    E_{p}
    \le
    \gamma
    \dnf(\dlov(k^{-r}E_{r}))
    \dlov(k^{-r}E_{r})^{p}
    \int_{M}
    \abs{\grad u}^{p}
    \di \msr
    +
    (2^{p-1}+1)
    k^{p-r}
    E_{r}
    \,.
  \end{equation}
  We select next $k$ so that
  \begin{equation}
    \label{eq:sgn_iv}
    \dnf(\dlov(k^{-r}E_{r}))
    \dlov(k^{-r}E_{r})^{p}
    \int_{M}
    \abs{\grad u}^{p}
    \di \msr
    =
    k^{p-r}
    E_{r}
    \,.
  \end{equation}
  Then we have first
  \begin{equation}
    \label{eq:sgn_v}
    E_{p}
    \le
    \gamma
    k^{p-r}
    E_{r}
    \,,
  \end{equation}
  yielding at once
  \begin{equation}
    \label{eq:sgn_vi}
    E_{r}
    k^{-r}
    \le
    \gamma S
    \,.
  \end{equation}
  On appealing again to \eqref{eq:sgn_iii}, \eqref{eq:sgn_iv}, together with \eqref{eq:sgn_vi} we obtain, recalling \eqref{eq:aux_vold_nnn},
  \begin{multline}
    \label{eq:sgn_vii}
    E_{p}
    \le
    \gamma
    \dnf(\dlov(\gamma S))
    \dlov(\gamma S)^{p}
    \int_{M}
    \abs{\grad u}^{p}
    \di \msr
    \\
    \le
    \gamma
    \dnf(\dlov( S))
    \dlov( S)^{p}
    \int_{M}
    \abs{\grad u}^{p}
    \di \msr
    \,.
  \end{multline}
\end{proof}

\begin{lemma}
  \label{l:sgns}
  Assume that $1<p<s<\SpDim$, and that the assumptions of Lemma~\ref{l:fks} hold true. Assume further that $0<r<p$.
  Then we have
  \begin{equation}
    \label{eq:sgns_n}
    E_{s}
    \le
    \gamma
    [
    \dnf(\dlov( \varSigma))
    \dlov( \varSigma)^{p}
    ]^{\frac{s}{p}}
    \varSigma^{1-\frac{s}{p}}
    \Big(
    \int_{M}
    \abs{\grad u}^{p}
    \di \msr
    \Big)^{\frac{s}{p}}
    \,,
  \end{equation}
  where $\varSigma=E_{r}^{s/(s-r)}E_{s}^{-r/(s-r)}$.
\end{lemma}

\begin{proof}
  As in the proof of Lemma~\ref{l:sgn}, we begin by splitting for a $k>0$ to be chosen,
    \begin{equation}
    \label{eq:sgns_j}
    E_{s}
    =
    \int_{\{\abs{u}>k\}}
    \dnf(d(x))
    \abs{u}^{s}
    \di \msr
    +
    \int_{\{\abs{u}\le k\}}
    \dnf(d(x))
    \abs{u}^{s}
    \di \msr
    =:
    J_{1}
    +
    J_{2}
    \,.
  \end{equation}
  Then we bound $J_{1}$ essentially as we did in \eqref{eq:sgn_i}, but applying \eqref{eq:fks_n} rather than \eqref{eq:fk_n}, and $J_{2}$ following the direct approach of \eqref{eq:sgn_ii}. We get, collecting such estimates, after a further use of Chebyshev inequality,
  \begin{multline}
    \label{eq:sgns_iii}
    E_{s}
    \le
    \gamma
    [
    \dnf(\dlov(k^{-r}E_{r}))
    \dlov(k^{-r}E_{r})^{p}
    ]^{\frac{s}{p}}
    (k^{-r}E_{r})^{1-\frac{s}{p}}
    \Big(
    \int_{M}
    \abs{\grad u}^{p}
    \di \msr
    \Big)^{\frac{s}{p}}
    \\
    +
    (2^{s-1}+1)
    k^{s-r}
    E_{r}
    \,.
  \end{multline}
  We have made use of \eqref{eq:fks_a} here.
  \\
  Next we choose $k$ so that
  \begin{equation}
    \label{eq:sgns_iv}
    [
    \dnf(\dlov(k^{-r}E_{r}))
    \dlov(k^{-r}E_{r})^{p}
    ]^{\frac{s}{p}}
    (k^{-r}E_{r})^{1-\frac{s}{p}}
    \Big(
    \int_{M}
    \abs{\grad u}^{p}
    \di \msr
    \Big)^{\frac{s}{p}}
    =
    k^{s-r}
    E_{r}
    \,,
  \end{equation}
  implying first that
  \begin{equation}
    \label{eq:sgns_v}
    E_{s}
    \le
    \gamma
    k^{s-r}
    E_{r}
    \,,
  \end{equation}
  so that
  \begin{equation}
    \label{eq:sgns_vi}
    E_{r}
    k^{-r}
    \le
    \gamma \varSigma
    \,.
  \end{equation}
  The proof then is concluded by substituting \eqref{eq:sgns_vi} in \eqref{eq:sgns_iii}, on recalling \eqref{eq:sgns_iv}, and \eqref{eq:fks_a}, too.
\end{proof}

\begin{corollary}
  \label{co:sgn_W}
  Under the assumptions and with the notation of Lemma~\ref{l:sgn} we have for the function $W$ defined in \eqref{eq:aux_function_n}
  \begin{equation}
    \label{eq:sgn_W_n}
    E_{p}
    \le
    \gamma
    E_{r}^{\frac{p(p-\alpha_{2})}{\Ha(r)}}
    W( S)^{\frac{(p-r)(\SpDim-\alpha_{1})}{\Ha(r)}}
    \Big(
    \int_{M}
    \abs{\grad u}^{p}
    \di \msr
    \Big)^{\frac{(p-r)(\SpDim-\alpha_{1})}{\Ha(r)}}
    \,,
  \end{equation}
  where $\Ha(r)=(p-r)(\SpDim-\alpha_{1})+r(p-\alpha_{2})$ for $r>0$.
\end{corollary}

\begin{proof}
  We only need to multiply and divide the right hand side of \eqref{eq:sgn_n} by $S^{(p-\alpha_{2})/(\SpDim-\alpha_{1})}$; then \eqref{eq:sgn_W_n} follows after some elementary algebra, when we recall the definitions of $S$ and of $W$.
\end{proof}

Again by simple algebra we infer also the following result.

\begin{corollary}
  \label{co:sgns_W}
  Under the assumptions and with the notation of Lemma~\ref{l:sgns} we have for the function $W$ defined in \eqref{eq:aux_function_n}
  \begin{equation}
    \label{eq:sgns_W_n}
    E_{s}
    \le
    \gamma
    E_{r}^{\frac{\Ha(s)}{\Ha(r)}}
    W( \varSigma)^{\frac{(s-r)(\SpDim-\alpha_{1})}{\Ha(r)}}
    \Big(
    \int_{M}
    \abs{\grad u}^{p}
    \di \msr
    \Big)^{\frac{(s-r)(\SpDim-\alpha_{1})}{\Ha(r)}}
    \,,
  \end{equation}
  where $\Ha$ is as in Corollary~\ref{co:sgn_W}.
\end{corollary}

We'll use the following inequalities. The set $C^{1}_{0}(M)$ denotes the $C^{1}(M)$ functions with compact support.

\begin{lemma}
  \label{l:cacc}
  Let $\unk$ be a solution of \eqref{eq:pde}--\eqref{eq:initd}, and
  let $\theta>0$, with $\theta> 2-m$ if $m<1$, $k>h>0$, $0<\tau_{2}<\tau_{1}$ be given. Then
  \begin{multline}
    \label{eq:cacc_n}
    \sup_{\tau_{1}<\tau<t}
    \int_{M}
    \ppos{\unk-k}^{1+\theta}
    \dnf
    \di\msr
    +
    \int_{\tau_{1}}^{t}
    \int_{M}
    \abs{\grad \ppos{\unk-k}^{\frac{p+m+\theta-2}{p}}}^{p}
    \di\msr
    \di\tau
    \\
    \le
    \gamma
    \frac{H(h,k)}{\tau_{1}-\tau_{2}}
    \int_{\tau_{2}}^{t}
    \int_{M}
    \ppos{\unk-h}^{1+\theta}
    \dnf
    \di\msr
    \di\tau
    \,,
  \end{multline}
  provided the right hand side in \eqref{eq:cacc_n} is finite. Here $H(h,k)=(k/(k-h))^{\pneg{m-1}}$.
\end{lemma}

\begin{lemma}
  \label{l:cacc2}
  Let $\unk$ be a solution of \eqref{eq:pde}--\eqref{eq:initd}, and
  let $\theta\ge p-1$. Let
  $\zeta\in C^{1}_{0}(M)$, $0\le \zeta\le 1$. Then
  \begin{multline}
    \label{eq:cacc2_n}
    \sup_{0<\tau<t}
    \int_{M}
    (\unk\zeta)^{1+\theta}
    \dnf
    \di\msr
    +
    \int_{0}^{t}
    \int_{M}
    \abs{\grad (\unk\zeta)^{\frac{p+m+\theta-2}{p}}}^{p}
    \di\msr
    \di\tau
    \\
    \le
    \gamma
    \Big\{
    \int_{0}^{t}
    \int_{M}
    \abs{\grad \zeta}^{p}
    \unk^{p+m+\theta-2}
    \di\msr
    \di\tau
    +
    \int_{M}
    (\unk_{0}\zeta)^{1+\theta}
    \di\msr
    \Big\}
    \,,
  \end{multline}
  provided the right hand side in \eqref{eq:cacc2_n} is finite.
\end{lemma}

The proofs of lemmas \ref{l:cacc} and \ref{l:cacc2} are standard and we omit them.

\section{Proof of the sup estimate in the case $p+m-3>0$.}
\label{s:sup}

For $k>0$ to be selected later, and for $\theta>0$ as in Lemma~\ref{l:cacc}, we define for $n\ge 0$ 
\begin{gather*}
  k_{n}
  =
  k
  (
  1
  -
  2\sigma
  +
  2^{-n}
  \sigma
  )
  \,,
  \quad
  r
  =
  \frac{p}{p+m+\theta-2}
  \,,
  \\
  s
  =
  r(1+\theta)
  <
  p
  \,,
  \quad
  v_{n}
  =
  \ppos{u-k_{n}}^{\frac{1}{r}}
  \,,
  \quad
  \tau_{n}
  =
  \frac{t}{2}
  (1-2\sigma+2^{-n}\sigma)
  \,.
\end{gather*}
Here $\sigma\in(1,1/4)$ is fixed. From Lemma~\ref{l:cacc} we infer
\begin{equation}
  \label{eq:sup_i}
  \sup_{\tau_{n}<\tau<t}
  \int_{M}
  \dnf
  v_{n}^{s}
  \di\msr
  +
  \int_{\tau_{n}}^{t}
  \int_{M}
  \abs{\grad v_{n}}^{p}
  \di\msr
  \di \tau
  \le
  \gamma
  \frac{2^{n\ell}}{\sigma^{\ell}t}
  \int_{\tau_{n+1}}^{t}
  \int_{M}
  \dnf
  v_{n+1}^{s}
  \di \msr
  \di \tau
  \,,
\end{equation}
for $\ell=1+\pneg{m-1}$. From H\"older inequality it follows
\begin{equation}
  \label{eq:sup_ii}
  \int_{M}
  \dnf
  v_{n+1}^{s}
  \di\msr
  \le
  \Big(
  \int_{M}
  \dnf
  v_{n+1}^{p}
  \di\msr
  \Big)^{\frac{s-r}{p-r}}
  \Big(
  \int_{M}
  \dnf
  v_{n+1}^{r}
  \di\msr
  \Big)^{\frac{p-s}{p-r}}
  \,,
\end{equation}
whence, on applying Corollary~\ref{co:sgn_W} to bound the first integral on the right hand side of \eqref{eq:sup_ii}, we get
\begin{multline}
  \label{eq:sup_iii}
  \int_{M}
  \dnf
  v_{n+1}^{s}
  \di\msr
  \le
  \gamma
  W(S)^{\frac{(s-r)(\SpDim-\alpha_{1})}{\Ha(r)}}
  \Big(
  \int_{M}
  \dnf
  v_{n+1}^{r}
  \di\msr
  \Big)^{\frac{p-s}{p-r}+\frac{(s-r)p(p-\alpha_{2})}{(p-r)\Ha(r)}}
  \\
  \times
  \Big(
  \int_{M}
  \abs{\grad v_{n+1}}^{p}
  \di\msr
  \Big)^{\frac{(s-r)(\SpDim-\alpha_{1})}{\Ha(r)}}
  \,.
\end{multline}
here $S$ is defined as in \eqref{eq:ES_not}. Note that by appealing again to H\"older inequality we obtain, with the notation introduced in \eqref{eq:measd}
\begin{equation}
  \label{eq:sup_iv}
  S
  \le
  \int_{\{v_{n+1}>0\}}
  \dnf
  \di\msr
  \le
  \measd(k_{\infty})
  =
  \msd(\{u(\tau)>k_{\infty}\})
  \,,
  \quad
  k_{\infty}
  =
  k(1-2\sigma)
  \,.
\end{equation}
Next we integrate in time the estimate \eqref{eq:sup_iii}, and after a further application of H\"older inequality we bound the right hand side of \eqref{eq:sup_i} by
\begin{equation}
  \label{eq:sup_v}
  \begin{split}
    &\gamma
    \frac{2^{n\ell}}{\sigma^{\ell}t}
    \int_{\tau_{n+1}}^{t}
    \Big(
    \int_{M}
    \dnf
    v_{n+1}^{r}
    \di\msr
    \Big)^{\frac{p-s}{p-r}+\frac{(s-r)p(p-\alpha_{2})}{(p-r)\Ha(r)}}
    \Big(
    \int_{M}
    \abs{\grad v_{n+1}}^{p}
    \di\msr
    \Big)^{\frac{(s-r)(\SpDim-\alpha_{1})}{\Ha(r)}}
    \di \tau
    \\
    &\quad
    \times
    \sup_{\tau_{n+1}<\tau<t}
    W(\{u(\tau)>k_{\infty}\})^{\frac{(s-r)(\SpDim-\alpha_{1})}{\Ha(r)}}
    \\
    &\le
    \frac{2^{n\ell}}{\sigma^{\ell}}
    t^{-\frac{(s-r)(\SpDim-\alpha_{1})}{\Ha(r)}}
    \Big(
    \int_{\tau_{n+1}}^{t}
    \int_{M}
    \abs{\grad v_{n+1}}^{p}
    \di\msr
    \di\tau
    \Big)^{\frac{(s-r)(\SpDim-\alpha_{1})}{\Ha(r)}}
    \\
    &\quad
    \times
    \sup_{\tau_{n+1}<\tau<t}
    \Big(
    \int_{M}
    \dnf
    v_{n+1}^{r}
    \di\msr
    \Big)^{\frac{\Ha(s)}{\Ha(r)}}
    \sup_{\tau_{n+1}<\tau<t}
    W(\{u(\tau)>k_{\infty}\})^{\frac{(s-r)(\SpDim-\alpha_{1})}{\Ha(r)}}
    \,.
  \end{split}
\end{equation}
Let us define for the sake of notational simplicity $\xi=\Ha(r)/[(s-r)(\SpDim-\alpha_{1})]$.
Next we invoke Young inequality to bound the right hand side of \eqref{eq:sup_v} by
\begin{equation}
  \label{eq:sup_vi}
  \begin{split}
    &\eps
    \int_{\tau_{n+1}}^{t}
    \int_{M}
    \abs{\grad v_{n+1}}^{p}
    \di\msr
    \di\tau
    +
    \gamma
    \frac{b^{n}}{\sigma^{\frac{\ell\xi}{\xi-1}}}
    \eps^{-\frac{1}{\xi-1}}
    t^{-\frac{(s-r)(\SpDim-\alpha_{1})}{(p-s)(\SpDim-\alpha_{1})+r(p-\alpha_{2})}}
    \\
    &\quad
    \times
    \sup_{\tau_{n+1}<\tau<t}
    W(\{u(\tau)>k_{\infty}\})^{\frac{(s-r)(\SpDim-\alpha_{1})}{(p-s)(\SpDim-\alpha_{1})+r(p-\alpha_{2})}}
    \\
    &\quad
    \times
    \sup_{\tau_{n+1}<\tau<t}
    \Big(
    \int_{M}
    \dnf
    v_{n+1}^{r}
    \di\msr
    \Big)^{\frac{\Ha(s)}{(p-s)(\SpDim-\alpha_{1})+r(p-\alpha_{2})}}
    \,,
  \end{split}
\end{equation}
where $b=2^{\ell\xi/(\xi-1)}$. On combining \eqref{eq:sup_i}--\eqref{eq:sup_vi}, and denoting
\begin{equation}
  \label{eq:sup_In}
  I_{n}
  =
  \sup_{\tau_{n}<\tau<t}
  \int_{M}
  \dnf
  v_{n}(\tau)^{s}
  \di\msr
  +
  \int_{\tau_{n}}^{t}
  \int_{M}
  \abs{\grad v_{n}}^{p}
  \di\msr
  \di\tau
  \,,
\end{equation}
we get the recursive inequality
\begin{equation}
  \label{eq:sup_kj}
  \begin{split}
    I_{n}
    &\le
    \eps
    I_{n+1}
    +
    \gamma
    \frac{b^{n}}{\sigma^{\frac{\ell\xi}{\xi-1}}}
    \eps^{-\frac{1}{\xi-1}}
    t^{-\frac{(s-r)(\SpDim-\alpha_{1})}{(p-s)(\SpDim-\alpha_{1})+r(p-\alpha_{2})}}
    \\
    &\quad
    \times
    \sup_{\tau_{\infty}<\tau<t}
    W(\{u(\tau)>k_{\infty}\})^{\frac{(s-r)(\SpDim-\alpha_{1})}{(p-s)(\SpDim-\alpha_{1})+r(p-\alpha_{2})}}
    \\
    &\quad
    \times
    \sup_{\tau_{\infty}<\tau<t}
    \Big(
    \int_{M}
    \dnf
    v_{\infty}^{r}
    \di\msr
    \Big)^{\frac{\Ha(s)}{(p-s)(\SpDim-\alpha_{1})+r(p-\alpha_{2})}}
    \,,
  \end{split}
\end{equation}
where $\tau_{\infty}=t(1-2\sigma)/2$, $v_{\infty}=\ppos{u-k_{\infty}}^{1/r}$.

A standard iteration process then implies that, when we select e.g., $\eps=1/(2b)$,
\begin{equation}
  \label{eq:sup_j}
  \begin{split}
    &\sup_{\tau_{0}<\tau<t}
    \int_{M}
    \dnf
    \ppos{u(\tau)-k_{0}}^{1+\theta}
    \di\msr
    \le
    \gamma
    \sigma^{-\frac{\ell\xi}{\xi-1}}
    t^{-\frac{\theta(\SpDim-\alpha_{1})}{\eta}}
    \\
    &\quad
    \times
    \sup_{\tau_{\infty}<\tau<t}
    W(\{u(\tau)>k_{\infty}\})^{\frac{\theta(\SpDim-\alpha_{1})}{\eta}}
    \\
    &\quad
    \times
    \sup_{\tau_{\infty}<\tau<t}
    \Big(
    \int_{M}
    \dnf
    \ppos{u(\tau)-k_{\infty}}
    \di\msr
    \Big)^{
      1
      +
      \frac{\theta(p-\alpha_{2})}{\eta}
    }
    \,,
  \end{split}
\end{equation}
where we have computed
\begin{gather*}
  \frac{(s-r)(\SpDim-\alpha_{1})}{(p-s)(\SpDim-\alpha_{1})+r(p-\alpha_{2})}
  =
  \frac{\theta(\SpDim-\alpha_{1})}{\eta}
  \,,
  \\
  \frac{\Ha(s)}{(p-s)(\SpDim-\alpha_{1})+r(p-\alpha_{2})}
  =
  1
  +
  \frac{\theta(p-\alpha_{2})}{\eta}
  \,,
  \\
  \eta
  =
  (\SpDim-\alpha_{1})(p+m-3)
  +
  p-\alpha_{2}
  \,.
\end{gather*}
Next we set in \eqref{eq:sup_j}
\begin{gather*}
  \tau_{\infty}
  =
  t_{j}
  =
  \frac{t}{2}
  (1-2^{-j-2})
  \,,
  \quad
  \tau_{0}
  =
  t_{j+1}
  \,,
  \quad
  \sigma
  =
  2^{-j-3}
  \,,
  \\
  k_{\infty}
  =
  h_{j}
  =
  k(1-2^{-j-2}
  )
  \,,
  \quad
  k_{0}
  =
  h_{j+1}
  \,,
  \\
  y_{j}
  =
  \sup_{t_{j}<\tau<t}
  \int_{M}
  \dnf
  \ppos{u(\tau)-h_{j}}
  \di\msr
  \,,
  \qquad
  j\ge 0
  \,.
\end{gather*}
Note that by Chebyshev inequality we obtain there
\begin{equation}
  \label{eq:sup_che1}
  W(\msd(\{u(\tau)>h_{j}\}))
  \le
  W\Big(
  2h_{j}^{-1}
  \int_{M}
  \dnf
  u(\tau)
  \di\msr
  \Big)
  \le
  \gamma
  W(k^{-1}U)
  \,,
\end{equation}
when we employ also Lemma~\ref{l:aux_function} and 
\begin{equation*}
  U
  :=
  \sup_{\frac{t}{2}<\tau<t}
  \int_{M}
  \dnf
  u(\tau)
  \di\msr
  \ge
  y_{j}
  \,,
  \qquad
  j\ge 0
  \,.
\end{equation*}
By the same token, the left hand side of \eqref{eq:sup_j} is bounded from below by
\begin{equation}
  \label{eq:sup_che2}
  k^{\theta}
  2^{-(j+4)\theta}
  y_{j+2}
  \,.
\end{equation}
Thus we infer for $Y_{i}=y_{2i}$, $i\ge 0$,
\begin{equation}
  \label{eq:sup_jj}
  Y_{i+1}
  \le
  \gamma
  4^{(\frac{\ell\xi}{\xi-1}+\theta)i}
  k^{-\theta}
  t^{-\frac{\theta(\SpDim-\alpha_{1})}{\eta}}
  W(k^{-1}U)^{\frac{\theta(\SpDim-\alpha_{1})}{\eta}}
  Y_{i}^{1+\frac{\theta(p-\alpha_{2})}{\eta}}
  \,.
\end{equation}
According to the classical result \cite[Lemma 5.6, Chapt. II]{LSU}, we get that $Y_{i}\to 0$ as $i\to+\infty$ provided
\begin{equation}
  \label{eq:sup_jv}
  k^{-1}
  t^{-\frac{\SpDim-\alpha_{1}}{\eta}}
  W(k^{-1}U)^{\frac{\SpDim-\alpha_{1}}{\eta}}
  U^{\frac{p-\alpha_{2}}{\eta}}
  \le
  \gamma_{0}
  \,,
\end{equation}
for a suitably small constant $\gamma_{0}>0$ depending on the parameters in \eqref{eq:sup_jj}.

When we recall the definition of the function $W$ we see that \eqref{eq:sup_jv} reduces to, when we use the bound for mass in \eqref{eq:mass_subcr_a},
\begin{equation}
  \label{eq:sup_vj}
  k^{-(p+m-3)}
  t^{-1}
  \dnf\big(\dlov(k^{-1}\norma{\dnf u_{0}}{1})\big)
  \dlov(k^{-1}\norma{\dnf u_{0}}{1})^{p}
  =
  \gamma_{0}
  \,.
\end{equation}
We conclude by invoking Remark~\ref{r:decayf}.

\section{Proof of the sup estimate in the case $p+m-3<0$.}
\label{s:sin}
We borrow the notation from Section~\ref{s:sup}, assuming also $\theta>3-p-m$ is so large that $r<p$, and that $s>p$ satisfies \eqref{eq:fks_nn}: see Remark~\ref{r:fks}. Let us also note explicitly that assumption \eqref{eq:sup_n}, with some elementary algebra, yields
\begin{equation}
  \label{eq:sin_iii}
  (\SpDim-\alpha_{1})(p-s)
  +
  r(p-\alpha_{2})
  >0
  \,.
\end{equation}
which in turn implies immediately $\Ha(s)>0$ ($\Ha$ has been defined in Corollary~\ref{co:sgn_W}).

We start again from \eqref{eq:sup_i}; on the right hand side there by means of \eqref{eq:sgns_W_n}, we get
\begin{multline}
  \label{eq:sin_i}
  \int_{M}
  \dnf
  v_{n+1}^{s}
  \di\msr
  \le
  \gamma
  \Big(
  \int_{M}
  \dnf
  v_{n+1}^{r}
  \di\msr
  \Big)^{\frac{\Ha(s)}{\Ha(r)}}
  \\
  \times
  W(\varSigma)^{\frac{(s-r)(\SpDim-\alpha_{1})}{\Ha(r)}}
  \Big(
  \int_{M}
  \abs{\grad v_{n+1}}^{p}
  \di\msr
  \Big)^{\frac{(s-r)(\SpDim-\alpha_{1})}{\Ha(r)}}
  \,,
\end{multline}
where $\varSigma$ is then bounded as in \eqref{eq:sup_iv}. From now on, the proof proceeds formally unchanged as in Section~\ref{s:sup}, with the remarks below.

We use H\"older inequality to bound the right hand side of
\eqref{eq:sup_i} exactly with the right hand side of \eqref{eq:sup_v}
(with the current values of $s$, $r$).  Note that the condition
\begin{equation}
  \label{eq:sin_ii}
  \frac{(s-r)(\SpDim-\alpha_{1})}{\Ha(r)}
  <
  1
  \,,
\end{equation}
required in order to apply H\"older (and, next, Young) inequality, is not
automatically satisfied as in the case of Section~\ref{s:sup}; anyway,
it is easily seen by a direct elementary calculation that
\eqref{eq:sin_ii} is equivalent to \eqref{eq:sin_iii}.

Then we apply Young inequality to arrive at \eqref{eq:sup_kj}. Here we specifically note that all exponents have the expected sign owing to \eqref{eq:sin_iii}, and that in \eqref{eq:sup_j} $\eta$ is positive by assumption \eqref{eq:sup_n}.

At the end of the proof, in order to define $k$ by means of \eqref{eq:sup_vj}, we exploit our assumption that $\fpsf$ is increasing.

\section{Proof of the finite speed of propagation}
\label{s:fsp}

Here we introduce the notation
\begin{gather*}
  A_{n}
  =
  \{
  x\in\ M
  \mid
  R'_{n}<d(x)<R''_{n}
  \}
  \,,
  \quad
  0<\eta\,,\sigma\le \frac{1}{4}
  \,,
  \quad
  R\ge 4 R_{0}
  \,,
  \\
  R'_{n}
  =
  \frac{R}{2}
  (
  1-\eta-\sigma+\sigma 2^{-n}
  )
  \,,
  \quad
  R''_{n}
  =
  \frac{R}{2}
  (
  1+\eta+\sigma-\sigma 2^{-n}
  )
  \,.
\end{gather*}
We also introduce a standard cutoff function $\zeta_{n}\in C^{1}_{0}(A_{n+1})$ such that
\begin{equation*}
  0\le\zeta_{n}\le 1
  \,;
  \quad
  \zeta_{n}(x)
  =
  1
  \,,
  \quad
  x\in A_{n}
  \,;
  \quad
  \abs{\grad \zeta_{n}(x)}
  \le
  \gamma
  2^{n}
  (\sigma R)^{-1}
  \,.
\end{equation*}
Then for $\theta>0$ as in Section~\ref{s:sup} we define
\begin{equation*}
  r
  =
  \frac{p}{p+m+\theta-2}
  \,,
  \quad
  s
  =
  (1+\theta)r
  <p
  \,,
  \quad
  v_{n}
  =
  (u\zeta_{n})^{\frac{p+m+\theta-2}{p}}
  \,.
\end{equation*}
Note that no $A_{n}$ intersects $\supp u_{0}$. Then from Lemma~\ref{l:cacc2} we get
\begin{multline}
  \label{eq:fsp_i}
  J_{n}
  :=
  \sup_{0<\tau<t}
  \int_{M}
  v_{n}^{s}
  \dnf
  \di\msr
  +
  \int_{0}^{t}
  \int_{M}
  \abs{\grad v_{n}}^{p}
  \di\msr
  \di\tau
  \\
  \le
  \gamma
  \frac{2^{np}}{\sigma^{p}R^{p}}
  \int_{0}^{t}
  \int_{M}
  v_{n+1}^{p}
  \di\msr
  \di\tau
  \,.
\end{multline}
Next we apply Corollary~\ref{co:emb_old_p} to $v_{n+1}$ and get for the present choice of $r<p$ (remember that the support of $v_{n+1}$ is contained in an annulus), and for all times, 
\begin{multline}
  \label{eq:fsp_ii}
  \int_{M}
  v_{n+1}^{p}
  \di \msr
  \le
  \gamma
  \ipf(\vol(R))^{\frac{p\SpDim(p-r)}{\SpDim(p-r)+rp}}
  \Big(
  \int_{M}
  v_{n+1}^{r}
  \di \msr
  \Big)^{\frac{p^{2}}{\SpDim(p-r)+rp}}
  \\
  \times
  \Big(
  \int_{M}
  \abs{\grad v_{n+1}}^{p}
  \di\msr
  \Big)^{\frac{\SpDim(p-r)}{\SpDim(p-r)+rp}}
  \,.
\end{multline}
We integrate in time this estimate, and apply Young inequality, to bound above the right hand side of \eqref{eq:fsp_i} with
\begin{multline}
  \label{eq:fsp_iii}
  \eps
  \int_{0}^{t}
  \abs{\grad v_{n+1}}^{p}
  \di\msr
  \\
  +
  \gamma
  \eps^{-\frac{\SpDim(p-r)}{rp}}
  t
  \Big(
  \frac{2^{n}}{\sigma R}
  \Big)^{\frac{\SpDim(p-r)+rp}{r}}
  \ipf(\vol(R))^{\frac{\SpDim(p-r)}{r}}
  \sup_{0<\tau<t}
  \Big(
  \int_{M}
  v_{n+1}^{r}
  \di\msr
  \Big)^{\frac{p}{r}}
  \,.
\end{multline}
From \eqref{eq:fsp_i} and \eqref{eq:fsp_iii} we infer, on substituting the definition of $r$, the recursive inequality
\begin{multline}
  \label{eq:fsp_iv}
  J_{n}
  \le
  \eps
  J_{n+1}
  +
  \gamma
  \eps^{-\frac{\SpDim}{p}(p+m+\theta-3)}
  \Big(
  \frac{2^{n}}{\sigma}
  \Big)^{\SpDim(p+m+\theta-3)+p}
  \\
  \times
  t
  \vol(R)^{-(p+m+\theta-3)}
  R^{-p}
  \sup_{0<\tau<t}
  \Big(
  \int_{A_{\infty}}
  u
  \di\msr
  \Big)^{p+m+\theta-2}
  \,,
\end{multline}
where
\begin{equation*}
  A_{\infty}
  =
  \Big\{
  x\in\ M
  \mid
  \frac{R}{2}
  (
  1-\eta-\sigma
  )
  <d(x)<
  \frac{R}{2}
  (
  1+\eta+\sigma
  )
  \Big\}
  \,.
\end{equation*}
A standard iterative argument, relying on a suitable choice of $\eps$, then yields the estimate
\begin{multline}
  \label{eq:fsp_v}
  \sup_{0<\tau<t}
  \int_{A_{0}}
  \dnf
  u^{1+\theta}
  \di\msr
  \le
  \gamma
  \sigma^{-\SpDim(p+m+\theta-3)-p}
  \\
  \times
  t
  \vol(R)^{-(p+m+\theta-3)}
  R^{-p}
  \sup_{0<\tau<t}
  \Big(
  \int_{A_{\infty}}
  u
  \di\msr
  \Big)^{p+m+\theta-2}
  \,.
\end{multline}
Define next a sequence of shrinking annuli
\begin{gather*}
  D_{n}
  =
  \{
  x\in M
  \mid
  \bar R'_{n}
  <d(x)<
  \bar R''_{n}
  \}
  \,,
  \\
  \bar R'_{n}
  =
  \frac{R}{2}(1-2^{-n-1})
  \,,
  \qquad
  \bar R''_{n}
  =
  \frac{R}{2}(1+2^{-n-1})
  \,.
\end{gather*}
We apply inequality \eqref{eq:fsp_v} with $A_{0}=D_{n+1}$, $A_{\infty}=D_{n}$, $\sigma=\eta=2^{-n-2}$, $n\ge 0$, and obtain, using the the fact that $\dnf(d(x))\ge \dnf(R)$ in $D_{n}$,
\begin{multline}
  \label{eq:fsp_vi}
  \sup_{0<\tau<t}
  \int_{D_{n+1}}
  \dnf
  u^{1+\theta}
  \di\msr
  \le
  \gamma
  (2^{n})^{\SpDim(p+m+\theta-3)-p}
  \\
  \times
  t
  \vol(R)^{-(p+m+\theta-3)}
  R^{-p}
  \dnf(R)^{-(p+m+\theta-2)}
  \sup_{0<\tau<t}
  \Big(
  \int_{D_{n}}
  \dnf
  u
  \di\msr
  \Big)^{p+m+\theta-2}
  \,.
\end{multline}
We get from H\"older inequality that
\begin{equation}
  \label{eq:fsp_vii}
  \begin{split}
    Y_{n}
    &:=
    \sup_{0<\tau<t}
    \int_{D_{n}}
    \dnf
    u
    \di\msr
    \le
    \Big(
    \int_{D_{n}}
    \dnf
    \di\msr
    \Big)^{\frac{\theta}{1+\theta}}
    \\
    &\quad
    \times
    \Big(
    \sup_{0<\tau<t}
    \int_{D_{n}}
    \dnf
    u^{1+\theta}
    \di\msr
    \Big)^{\frac{1}{1+\theta}}
    \\
    &\le
    \gamma
    \big(
    \vol(R)\dnf(R)
    \big)^{\frac{\theta}{1+\theta}}
    \Big(
    \sup_{0<\tau<t}
    \int_{D_{n}}
    \dnf
    u^{1+\theta}
    \di\msr
    \Big)^{\frac{1}{1+\theta}}
    \,,
  \end{split}
\end{equation}
where we also took into account \eqref{eq:fsp_m}.

From \eqref{eq:fsp_vi} and \eqref{eq:fsp_vii} we infer at once
\begin{equation}
  \label{eq:fsp_viii}
  Y_{n+1}
  \le
  \gamma
  b^{n}
  \Big(
  \frac{
    t
  }{
    \vol(R)^{p+m-3}
    R^{p}
    \dnf(R)^{p+m-2}
  }
  \Big)^{\frac{1}{1+\theta}}
  Y_{n}^{1+\frac{p+m-3}{1+\theta}}
  \,,
\end{equation}
where $b$ is a suitable power of $2$. It follows from \cite[Lemma 5.6, Chapt. II]{LSU} that $Y_{n}\to 0$ as $n\to\infty$ if $R$ is chosen so that
\begin{equation}
  \label{eq:fsp_ix}
  \frac{
    t
  }{
    \vol(R)^{p+m-3}
    R^{p}
    \dnf(R)^{p+m-2}
  }
  \Big(
  \int_{M}
  \dnf
  u_{0}
  \di\msr
  \Big)^{p+m-3}
  \le
  \gamma_{0}
  \,,
\end{equation}
for a suitable constant $\gamma_{0}>0$ depending on the parameters of the problem. We also use here the bound in \eqref{eq:mass_subcr_a}.

Finally, note that according to the definition of $Y_{n}$ we proved that $u(x,t)=0$ for $x\in M\setminus B_{R}$ if $R$ satisfies \eqref{eq:fsp_ix} and of course the condition $\supp u_{0}\subset B_{R/4}$ stated at the beginning of the proof. We have thus proved the sought after result.

\section{Proof of the universal bound}
\label{s:unb}

We need the following

\begin{lemma}
  \label{l:unb_emb}
  Assume that the assumptions of Lemma~\ref{l:sob}, \eqref{eq:M_inc}, \eqref{eq:M_isoup} hold true, and that for $\alpha>0$, $\beta\in(0,\SpDim)$ we have for a suitable $c>0$
  \begin{equation}
    \label{eq:unb_emb_dnf}
    \dnf(\tau)
    \le
    c^{-1} \tau^{-\alpha}
    \,,
    \quad
    \tau>1
    \,,
  \end{equation}
  and
  \begin{equation}
    \label{eq:unb_emb_vol}
    \vol(\tau)
    \ge
    c
    \tau^{\beta}
    \,,
    \quad
    \tau>1
    \,.
  \end{equation}
  In addition we require that
  one of the following holds:
  \begin{gather*}
    p<\beta
    \,,
    \quad
    \alpha \ge \beta
    \,,
    \quad
    0<r<p^{*}
    \,,
    \\
    p<\beta
    \,,
    \quad
    p\frac{\SpDim-\beta}{\SpDim-p}
    <\alpha < \beta
    \,,
    \quad
    p\frac{\beta-\alpha}{\beta-p}
    <
    r<p^{*}
    \,,
    \\
    p=\beta
    \,,
    \quad
    \alpha > \beta
    \,,
    \quad
    0<r<p^{*}
    \,,
    \\
    p>\beta
    \,,
    \quad
    0<r<\min \Big(
    p^{*}
    ,
    p
    \frac{\alpha-\beta}{p-\beta}
    \Big)
    \,.
  \end{gather*}
  Then
  \begin{equation}
    \label{eq:unb_emb_n}
    \int_{M}
    \dnf
    \abs{u}^{r}
    \di\msr
    \le
    \gamma
    \Big(
    \int_{M}
    \abs{\grad u}^{p}
    \di\msr
    \Big)^{\frac{r}{p}}
    \,.
  \end{equation}
\end{lemma}

\begin{proof}
  From H\"older inequality and Lemma~\ref{l:sob} we infer
  \begin{multline}
    \label{eq:unb_emb_i}
    \int_{M}
    \dnf
    \abs{u}^{r}
    \di\msr
    \le
    \Big(
    \int_{M}
    \abs{u}^{p^{*}}
    \ipf(\vol(d(x)))^{-p^{*}}
    \di\msr
    \Big)^{\frac{r}{p^{*}}}
    J^{1-\frac{r}{p^{*}}}
    \\
    \le
    \gamma
    \Big(
    \int_{M}
    \abs{\grad u}^{p}
    \di\msr
    \Big)^{\frac{r}{p}}
    J^{1-\frac{r}{p^{*}}}
    \,,
  \end{multline}
  where
  \begin{multline}
    \label{eq:unb_emb_ii}
    J
    :=
    \int_{M}
    \dnf(d(x))^{\frac{p^{*}}{p^{*}-r}}
    \ipf(\vol(d(x)))^{\frac{rp^{*}}{p^{*}-r}}
    \di \msr
    \\
    \le
    \gamma
    +
    \int_{M\setminus B_{1}}
    \dnf(d(x))^{\frac{p^{*}}{p^{*}-r}}
    \ipf(\vol(d(x)))^{\frac{rp^{*}}{p^{*}-r}}
    \di \msr
    \,.
  \end{multline}
  We majorize the last integral as
  \begin{multline}
    \label{eq:unb_emb_iii}
    \int_{1}^{+\infty}
    \dnf(\tau)^{\frac{p^{*}}{p^{*}-r}}
    \ipf(\vol(\tau))^{\frac{rp^{*}}{p^{*}-r}}
    \der{\vol}{\tau}(\tau)
    \di\tau
    \\
    \le
    \gamma
    \int_{1}^{+\infty}
    \dnf(\tau)^{\frac{p^{*}}{p^{*}-r}}
    \tau^{\frac{rp^{*}}{p^{*}-r}-1}
    \vol(\tau)^{1-\frac{rp^{*}}{\SpDim(p^{*}-r)}}
    \di\tau
    \,,
  \end{multline}
  where we have used the definition of $\ipf$ and \eqref{eq:M_inc}, \eqref{eq:M_isoup}. Finally, the right hand side of \eqref{eq:unb_emb_iii} is bounded by
  \begin{equation}
    \label{eq:unb_emb_iv}
    \int_{1}^{+\infty}
    \tau^{\frac{p^{*}(r-\alpha)}{p^{*}-r}-1+\beta-\frac{\beta rp^{*}}{\SpDim(p^{*}-r)}}
    \di\tau
    \,,
  \end{equation}
  according to our assumptions on $\vol$ and $\dnf$. The last integral converges in the cases given in the statement, as a direct inspection shows.
\end{proof}

\begin{proof}[Proof of Theorem]
  We use the notation introduced in Section~\ref{s:sup}, with the exception that while we still denote $s=p(1+\theta)/(p+m+\theta-2)<p$, $\theta>0$, we choose $p<r<p^{*}$. We may apply Lemma~\ref{l:unb_emb}, since \eqref{eq:unb_emb_vol} with $\beta=p$ follows under our assumptions from \eqref{eq:aux_vold_n}, so that $p<r<p^{*}$ is in the admissible range for Lemma~\ref{l:unb_emb}. Then we have at every time level, since $s<p<r$,
  \begin{multline}
    \label{eq:unb_i}
    \int_{M}
    \dnf
    v_{n+1}^{s}
    \di\msr
    \le
    \Big(
    \int_{M}
    \dnf
    v_{n+1}^{r}
    \di\msr
    \Big)^{\frac{s}{r}}
    \msd(\{u>k_{n+1}\})^{1-\frac{s}{r}}
    \\
    \le
    \gamma
    \Big(
    \int_{M}
    \abs{\grad v_{n+1}}^{p}
    \di\msr
    \Big)^{\frac{s}{p}}
    \msd(\{u>k_{\infty}\})^{1-\frac{s}{r}}
    \,.
  \end{multline}
  Then we apply to \eqref{eq:sup_i} the estimate \eqref{eq:unb_i} and Young inequality; with the notation \eqref{eq:sup_In}, we arrive at
  \begin{multline}
    \label{eq:unb_ii}
    I_{n}
    \le
    \eps
    \int_{\tau_{n+1}}^{t}
    \int_{M}
    \abs{\grad v_{n+1}}^{p}
    \di\msr
    \di\tau
    \\
    +
    \gamma
    \eps^{-\frac{s}{p-s}}
    \frac{2^{\frac{n\ell p}{p-s}}}{\sigma^{\frac{\ell p}{p-s}}}
    t^{-\frac{s}{p-s}}
    \sup_{\tau_{\infty}<\tau<t}
    \msd(\{u(\tau)>k_{\infty}\})^{\frac{p(r-s)}{r(p-s)}}
    \,.
  \end{multline}
  Then an iterative argument essentially identical to the one employed in \eqref{eq:sup_vi}--\eqref{eq:sup_kj} leads us, for a suitable choice of $\eps>0$, to
  \begin{multline}
    \label{eq:unb_iv}
    \sup_{\tau_{0}<\tau<t}
    \int_{M}
    \dnf
    \ppos{u(\tau)-k_{0}}^{1+\theta}
    \di\msr
    \\
    \le
    \gamma
    \sigma^{-\frac{\ell (p+m+\theta-2)}{p+m-3}}
    t^{-\frac{1+\theta}{p+m-3}}
    \sup_{\tau_{\infty}<\tau<t}
    \msd(\{u(\tau)>k_{\infty}\})^{1+\frac{(r-p)(1+\theta)}{r(p+m-3)}}
    \,.
  \end{multline}
  Note that the last exponent is greater than $1$ owing to our choice $r>p$.
  
  We proceed as in Section~\ref{s:sup}, introducing $h_{j}$, $y_{j}$ and $Y_{i}=y_{2i}$ as there. When we take into account that
  \begin{equation*}
    \msd(\{u(\tau)>h_{j}\})
    \le
    2^{j+2}
    k^{-1}
    y_{j-1}
    \,,
  \end{equation*}
  we get (remember that we set $k_{\infty}=h_{j}$, $k_{0}=h_{j+1}$)
  \begin{equation}
    \label{eq:unb_v}
    Y_{i+1}
    \le
    \gamma
    4^{\big(\frac{\ell (p+m+\theta-2)}{p+m-3}+\theta+1\big)i}
    k^{-(\theta +1)\big( 1 + \frac{r-p}{r(p+m-3)}\big)}
    t^{-\frac{1+\theta}{p+m-3}}
    Y_{i}^{1+\frac{(r-p)(1+\theta)}{r(p+m-3)}}
    \,,
  \end{equation}
  for $i\ge 0$. It follows from \cite[Lemma 5.6, Chapt. II]{LSU} that $Y_{i}\to 0$ as $i\to+\infty$, i.e., $u(t)\le k$, if
  \begin{equation}
    \label{eq:unb_vi}
    k^{-1-\frac{r-p}{r(p+m-3)}}
    t^{-\frac{1}{p+m-3}}
    Y_{0}^{\frac{r-p}{r(p+m-3)}}
    \le
    \gamma_{0}
    \,,
  \end{equation}
  for a suitable constant $\gamma_{0}>0$ depending on the parameters. We bound the integral appearing in $Y_{0}$ (at each time level) as
  \begin{equation*}
    \begin{split}
      &\int_{M}
      \dnf
      \Big(u-\frac{3}{4}k\Big)_{+}
      \di\msr
      \le
      \Big(
      \int_{M}
      \dnf
      \Big(u-\frac{3}{4}k\Big)_{+}^{q+1}
      \di\msr
      \Big)^{\frac{1}{q+1}}
      \msd\Big(\Big\{u>\frac{3}{4}k\Big\}\Big)^{\frac{q}{q+1}}
      \\
      &\quad
      \le
      \Big(
      \int_{M}
      \dnf
      \Big(u-\frac{3}{4}k\Big)_{+}^{q+1}
      \di\msr
      \Big)^{\frac{1}{q+1}}
      \Big(
      (4k^{-1})^{q+1}
      \int_{M}
      \dnf
      \Big(u-\frac{k}{2}\Big)_{+}^{q+1}
      \di\msr
      \Big)^{\frac{q}{q+1}}
      \\
      &\quad
      \le
      (
      4k^{-1}
      )^{q}
      \int_{M}
      \dnf
      u^{q+1}
      \di\msr
      \,,
    \end{split}
  \end{equation*}
  for $q>0$. It follows then from \eqref{eq:unb_vi} (with an equality) that
  \begin{equation}
    \label{eq:unb_j}
    \norma{u(t)}{\infty}
    \le
    \gamma
    t^{-\frac{r}{\Ka}}
    \sup_{t/4<\tau<t}
    E_{q+1}(\tau)^{\frac{r-p}{\Ka}}
    \,,
    \qquad
    t>0
    \,,
  \end{equation}
  for $E_{q+1}$ defined as in \eqref{eq:ES_not}, and
  \begin{equation*}
    \Ka
    =
    r(p+m-3)
    +
    (r-p)(q+1)
    \,.
  \end{equation*}
  We are left with the task of estimating $E_{q+1}(\tau)$; this will be accomplished by appealing again to Lemma~\ref{l:unb_emb}, where we select
  \begin{equation}
    \label{eq:unb_jj}
    0
    <
    r'
    :=
    \frac{p(1+q)}{p+m+q-2}
    <
    p
    <
    p^{*}
    \,.
  \end{equation}
  We obtain from the differential equation \eqref{eq:pde}, for $w=u^{(p+m+q-2)/p}$ the equality in
  \begin{multline}
    \label{eq:unb_jjj}
    \frac{1}{q+1}
    \der{E_{q+1}}{t}
    =
    -
    \Big(
    \frac{p}{p+m+\theta-2}
    \Big)^{p}
    \int_{M}
    \abs{\grad w}^{p}
    \di\msr
    \le
    -
    \gamma
    E_{q+1}(t)^{\frac{p+m+q-2}{1+q}}
    \,,
  \end{multline}
  where the inequality follows from an application of \eqref{eq:unb_emb_n} with $r$ replaced with the $r'$ given in \eqref{eq:unb_jj}. On integrating \eqref{eq:unb_jjj} we get
  \begin{equation}
    \label{eq:unb_jv}
    E_{q+1}(t)
    \le
    \gamma
    t^{-\frac{1+q}{p+m-3}}
    \,,
    \qquad
    t>0
    \,;
  \end{equation}
  actually we integrate over $(t_{0},t)$ and then let $t_{0}\to 0+$ in order to circumvent possible problems with the local summability of the initial data. Finally we substitute \eqref{eq:unb_jv} in \eqref{eq:unb_j} to prove the claim of the Theorem.
\end{proof}

\section{Interface blow up}
\label{s:ibl}

Let us reason by contradiction, and assume that the support of $u(t)$ stays bounded for all times.

Then we get for all times and for a fixed $\theta>0$, by means of H\"older and Hardy inequality
\begin{multline}
  \label{eq:ibl_i}
  \int_{M}
  \dnf
  u
  \di\msr
  \le
  \Big(
  \int_{M}
  d(x)^{-p}
  u^{p+m+\theta-2}
  \di \msr
  \Big)^{\frac{1}{p+m+\theta-2}}
  I(\theta)^{\frac{p+m+\theta-3}{p+m+\theta-2}}
  \\
  \le
  \gamma
  \Big(
  \int_{M}
  \abs{\grad u^{\frac{p+m+\theta-2}{p}}}^{p}
  \di \msr
  \Big)^{\frac{1}{p+m+\theta-2}}
  I(\theta)^{\frac{p+m+\theta-3}{p+m+\theta-2}}
  \,,
\end{multline}
where, when we recall \eqref{eq:M_inc}
\begin{equation*}
  \begin{split}
  I(\theta)
  &=
  \int_{M}
  \big(
  d(x)^{p}
  \dnf(x)^{p+m+\theta-2}
  \big)^{\frac{1}{p+m+\theta-3}}
  \di\msr
  \\
  &\le
  \SpDim
  \int_{0}^{+\infty}
  \big(
  \tau^{p}
  \dnf(\tau)^{p+m+\theta-2}
  \big)^{\frac{1}{p+m+\theta-3}}
  \frac{\vol(\tau)}{\tau}
  \di\tau
  \\
  &=
  \SpDim
  \int_{0}^{+\infty}
  \big(
  \tau^{p}
  \dnf(\tau)
  \big)^{-\frac{\theta}{(p+m-3)(p+m+\theta-3)}}
  \fpsf(\tau)^{\frac{1}{p+m-3}}
  \frac{\di\tau}{\tau}
  <
  +\infty
  \,,
  \end{split}
\end{equation*}
by assumption \eqref{eq:ibl_n}, if $\theta$ is chosen suitably small.

By means of calculation in all similar we get
\begin{multline}
  \label{eq:ibl_ii}
  \int_{M}
  \dnf
  u^{1+\theta}
  \di\msr
  \le
  \Big(
  \int_{M}
  d(x)^{-p}
  u^{p+m+\theta-2}
  \di \msr
  \Big)^{\frac{1+\theta}{p+m+\theta-2}}
  J(\theta)^{\frac{p+m-3}{p+m+\theta-2}}
  \\
  \le
  \gamma
  \Big(
  \int_{M}
  \abs{\grad u^{\frac{p+m+\theta-2}{p}}}^{p}
  \di \msr
  \Big)^{\frac{1+\theta}{p+m+\theta-2}}
  J(\theta)^{\frac{p+m-3}{p+m+\theta-2}}
  \,,
\end{multline}
where
when we recall \eqref{eq:M_inc}
\begin{equation*}
  \begin{split}
  J(\theta)
  &=
  \int_{M}
  \big(
  d(x)^{p(1+\theta)}
  \dnf(x)^{p+m+\theta-2}
  \big)^{\frac{1}{p+m-3}}
  \di\msr
  \\
  &\le
  \SpDim
  \int_{0}^{+\infty}
  \big(
  \tau^{p}
  \dnf(\tau)^{p+m+\theta-2}
  \big)^{\frac{1}{p+m-3}}
  \frac{\vol(\tau)}{\tau}
  \di\tau
  \\
  &=
  \SpDim
  \int_{0}^{+\infty}
  \big(
  \tau^{p}
  \dnf(\tau)
  \big)^{\frac{\theta}{p+m-3}}
  \fpsf(\tau)^{\frac{1}{p+m-3}}
  \frac{\di \tau)}{\tau}
  <
  +\infty
  \,,
  \end{split}
\end{equation*}
again by assumption \eqref{eq:ibl_n}, for a suitable choice of $\theta$.

On using \eqref{eq:ibl_ii} and the differential equation \eqref{eq:pde}, we obtain
\begin{multline}
  \label{eq:ibl_iii}
  \frac{1}{1+\theta}
  \der{}{t}
  \int_{M}
  \dnf
  u(t)^{1+\theta}
  \di\msr
  =
  -
  \Big(
  \frac{p}{p+m+\theta-2}
  \Big)^{p}
  \int_{M}
  \abs{\grad u^{\frac{p+m+\theta-2}{p}}}^{p}
  \di\msr
  \\
  \le
  -
  \gamma
  \Big(
  \int_{M}
  \dnf
  u(t)^{1+\theta}
  \di\msr
  \Big)^{\frac{p+m+\theta-2}{1+\theta}}
  \,,
\end{multline}
whence
\begin{equation}
  \label{eq:ibl_iv}
  \int_{M}
  \dnf
  u(t)^{1+\theta}
  \di\msr
  \le
  \gamma
  t^{-\frac{1+\theta}{p+m-3}}
  \,,
  \qquad
  t>0
  \,.
\end{equation}
However, \eqref{eq:ibl_i} and H\"older inequality yield
\begin{equation}
  \label{eq:ibl_v}
  \int_{t}^{t+1}
  \int_{M}
  \dnf
  u
  \di\msr
  \di\tau
  \le
  \Big(
  \int_{t}^{t+1}
  \int_{M}
  \abs{\grad u^{\frac{p+m+\theta-2}{p}}}^{p}
  \di\msr
  \di\tau
  \Big)^{\frac{1}{p+m+\theta-2}}
  \,.
\end{equation}
Again integrating the equality in \eqref{eq:ibl_iii}, we get
\begin{equation}
  \label{eq:ibl_vi}
  \int_{t}^{t+1}
  \int_{M}
  \abs{\grad u^{\frac{p+m+\theta-2}{p}}}^{p}
  \di\msr
  \di\tau
  \le
  \gamma
  \int_{M}
  \dnf
  u(t)^{1+\theta}
  \di\msr
  \,,
\end{equation}
and finally on combining \eqref{eq:ibl_iv} with \eqref{eq:ibl_vi} we obtain
\begin{equation}
  \label{eq:ibl_vii}
  \int_{M}
  \dnf
  u_{0}
  \di\msr
  =
  \int_{t}^{t+1}
  \int_{M}
  \dnf
  u(t)
  \di\msr
  \di\tau
  \le
  \gamma
  t^{-\frac{1+\theta}{(p+m+3)(p+m+\theta-2)}}
  \,,
\end{equation}
at least for all the times $t$ such that the compactness of the support  of $u(\tau)$ holds true over $(0,t+1)$. Indeed it is known, as we recalled in Remark~\ref{r:mass_subcr}, that such a property implies conservation of mass. But, \eqref{eq:ibl_vii} clearly is inconsistent as $t\to+\infty$, completing our argument.

\bibliographystyle{abbrv}
\bibliography{paraboli,pubblicazioni_andreucci}
\end{document}